\documentclass{article}

\usepackage[final]{neurips_2019}

\usepackage[utf8]{inputenc} \usepackage[T1]{fontenc}    \usepackage{hyperref}       \usepackage{url}            \usepackage{booktabs}       \usepackage{amsfonts}       \usepackage{nicefrac}       \usepackage{microtype}

\usepackage{pdfpages}
\usepackage{microtype}
\usepackage{graphicx}
\usepackage{subfigure}

\usepackage{amssymb}
\usepackage{amsmath}
\usepackage{amsthm}
\usepackage{mathtools}
\usepackage{subfiles}

\usepackage{multirow}
\usepackage{makecell}
\usepackage{paralist}
\usepackage{wrapfig}

\usepackage{algorithm}
\usepackage{algpseudocode}
\newtheorem{definition}{Definition}
\newtheorem{lemma}{Lemma}

\newtheorem{theorem}{Theorem}

\newtheorem{fact*}{Fact}
\newtheorem{claim}{Claim}

\DeclarePairedDelimiter{\norm}{\lVert}{\rVert}

\DeclareMathOperator*{\argmax}{arg\,max}
\DeclareMathOperator*{\argmin}{arg\,min}

\usepackage{xcolor}
\usepackage{soul}
\setstcolor{blue}

\newcommand{\measure}[1]{\mathcal{P}\left(#1\right)}

\newcommand{\derivativeclosed}[1]{#1}\newcommand{\derivative}[1]{#1}
\newcommand{\lemref}[1]{Lemma \ref{#1}}
\newcommand{\theoref}[1]{Theorem \ref{#1}}

\renewcommand{\eqref}[1]{Equation \ref{#1}}
\usepackage{fancybox}

\title{
Poincar\'{e} Recurrence, Cycles and Spurious Equilibria in Gradient-Descent-Ascent for
Non-Convex Non-Concave Zero-Sum Games
}

\author{Lampros Flokas\thanks{Equal contribution} \\
  Department of Computer Science\\
  Columbia University\\
  New York, NY 10025 \\
  \texttt{lamflokas@cs.columbia.edu} \\
\And
  Emmanouil V. Vlatakis-Gkaragkounis\footnotemark[1] \\
  Department of Computer Science\\
  Columbia University\\
  New York, NY 10025 \\
  \texttt{emvlatakis@cs.columbia.edu} \\
  \AND
  Georgios Piliouras \\
  Engineering Systems and Design \\
  Singapore University of Technology and Design\\
  Singapore \\
  \texttt{georgios@sutd.edu.sg} \\
}

\begin{document}

\maketitle

\begin{abstract}
     We study a wide class of non-convex non-concave min-max games that generalizes over standard bilinear zero-sum games. In this class, players control the inputs of a smooth function whose output is being applied to a bilinear zero-sum game. This class of games is motivated by the  indirect nature of the competition in
     Generative Adversarial Networks, where players control the parameters of a neural network while the actual competition happens between the distributions that the generator and discriminator capture. We establish theoretically, that depending on the specific instance of the problem 
     gradient-descent-ascent dynamics  can exhibit a variety of behaviors antithetical to convergence to the game theoretically meaningful min-max solution. Specifically, 
     different forms of recurrent behavior (including periodicity and Poincar\'{e} recurrence) are possible as well as convergence to spurious (non-min-max) equilibria for a positive measure of initial conditions. At the technical level, our analysis combines tools from optimization theory, game theory and dynamical systems.    
\end{abstract}

    \section{Introduction}
Min-max optimization is a problem of interest in several communities including Optimization, Game Theory and Machine Learning. In its most general form, given an objective function $r: \mathbb{R}^{n}\times \mathbb{R}^{m} \to \mathbb{R}$ and we would like to solve the following problem
\begin{equation} \label{eq:min-max}
(\pmb{\theta}^*,\pmb{\phi}^*) = \argmin_{\pmb{\theta} \in \mathbb{R}^{n} }\argmax_{\pmb{\phi} \in \mathbb{R}^{m} } r(\pmb{\theta},\pmb{\phi}).
\end{equation}
This problem is much more complicated compared to classical minimization problems, as even understanding under which conditions such a solution is meaning-full is far from trivial \cite{daskalakis2018limit,mai2017rock,oliehoek2018beyond,jin2019minmax}. What is even more demanding is understanding what kind of algorithms/dynamics are able to solve this problem when a solution is well defined.
    
Recently this problem has attracted renewed interest motivated by the advent of Generative Adversarial Networks (GANs) and their numerous applications \cite{gan,radford2015unsupervised,isola2017image,gan,zhang2017stackgan,arjovsky2017wasserstein,ledig2017photo,salimans2016improved}. A classical GAN architecture mainly revolves around the competition between two players, the generator and the discriminator. On the one hand, the generator aims to train a neural network based generative model that can generate high fidelity samples from a target distribution. On the other hand, the discriminator's goal is to train a neural network classifier than can distinguish between the samples of the target distribution and artificially generated samples. While one could consider each of the tasks in isolation, it is the competitive interaction between the generator and the discriminator that has lead to the resounding success of GANs. It is the "criticism" from a powerful discriminator that pushes the generator to capture the target distribution more accurately and it is the access to high fidelity artificial samples from a good generator that gives rise to better discriminators. Machine Learning researchers and practitioners have tried to formalize this competition using the min-max optimization framework mentioned above with great success \cite{arora2017generalization,ma2018generalization,ge2018fictitious,yazici2018unusual}.
    
One of the main limitations of this framework however is that to this day efficiently training GANs can be a notoriously difficult task \cite{salimans2016improved,metz2016unrolled,mertikopoulos2018cycles,kodali2017convergence}. Addressing this limitation has been the object of interest for a long line work in the recent years \cite{mescheder2018training,metz2016unrolled,pfau2016connecting,radford2015unsupervised,tolstikhin2017adagan,berthelot2017began,gulrajani2017improved}. Despite the intensified study, very little is known about efficiently solving general min-max optimization problems. Even for the relatively simple case of bilinear games, the little results that are known have usually a negative flavour. For example, the continuous time analogue of standard game dynamics such as gradient-descent-ascent or multiplicative weights lead to cyclic or recurrent behavior \cite{piliouras2014optimization,mertikopoulos2018cycles} whereas when they are actually run in discrete-time\footnote{Interestingly, running alternating gradient-descent-ascent in discrete-time results once again in recurrent behavior \cite{bailey2019finite}.} they lead to divergence and chaos \cite{BaileyEC18,cheung2019vortices,bailey2019fast}.
While positive results for the case of bilinear games exist, like extra-gradient (optimistic) training (\cite{daskalakis2017training,mertikopoulos2018mirror,daskalakis2018last}) and other techniques \cite{balduzzi2018mechanics,gidel2019negative,gidel2019a,abernethy2019lastiterate},
these results fail to generalize to  complex non-convex non-concave settings \cite{oliehoek2018beyond, lin2018solving, sanjabi2018solving}. 
In fact, for the case of non-convex-concave optimization, game theoretic interpretations of equilibria might not even be meaningful \cite{mazumdar2018convergence,jin2019minmax,adolphs2018local}.
    
In order to shed some light to this intellectually challenging problem, we propose a quite general class of min-max optimization problems that includes bilinear games as well as a wide range of non-convex non-concave games. In this class of problems, each player submits its own decision vector just like in general min-max optimization problems. Then each decision vector is processed separately by a (potentially different) smooth function. Each player finally gets rewarded by plugging in the processed decision vectors to a simple bilinear game.     More concretely, there are functions $F: \mathbb{R}^{n} \to \mathbb{R}^N$ and $G: \mathbb{R}^{m} \to \mathbb{R}^{M}$ and a matrix  $U_{N  \times M}$ such that
\begin{equation}
    r(\pmb{\theta},\pmb{\phi}) = \pmb{F}(\pmb{\theta})^\top U \pmb{G}(\pmb{\phi}).
\end{equation}
We call the resulting class of problems Hidden Bilinear Games. 
    
The motivation behind the proposed class of gamess is actually the setting of training GANs itself. During the training process of GANs, the discriminator and the generator "submit" the parameters of their corresponding neural network architectures, denoted as  $\pmb{\theta}$ and $\pmb{\phi}$ in our problem formulation. However, deep networks introduce nonlinearities in mapping their parameters to their output space which we capture through the non-convex functions $F, G$.     
Thus, even though hidden bilinear games do not demonstrate the full complexity of modern GAN architectures and training, they manage to capture two of its most pervasive properties:  \emph{i) the indirect competition of the generator and the discriminator}  
    and \emph{ii) the non-convex non-concave nature of training GANs}. Both features are markedly missing from simple bilinear games.

{\bf Our results.} We provide, the first to our own knowledge, global analysis  of gradient-descent-ascent for a class of non-convex non-concave zero-sum games that by design includes both features of bilinear zero-sum games as well as of single-agent non-convex optimization. Our analysis focuses on the (smoother) continuous time dynamics (Section \ref{section:two-by-two},\ref{section:general}) but we also discuss the implications for discrete time (Section \ref{section:discrete}). The unified thread of our results is that gradient-descent-ascent can exhibit a variety of behaviors antithetical to convergence to the min-max solution. In fact, convergence to a set of parameters that implement the desired min-max solution (as e.g. GANs require), if it actually happens, is more of an accident due to fortuitous system initialization rather than an implication of the adversarial network architecture.

Informally, we prove that these dynamics exhibit conservation laws, akin to energy conservation in physics. Thus, in contrast to them making progress over time their natural tendencies is to "cycle" through their parameter space. If the hidden bilinear game $U$ is 2x2 (e.g. Matching Pennies) with an interior Nash equilibrium, then the behavior is typically periodic (Theorem \ref{theorem:orbit}). If it is a higher dimensional game (e.g. akin to Rock-Paper-Scissors) then even more complex behavior is possible. Specifically, the system is formally analogous to Poincar\'{e} recurrent systems (e.g. many body problem in physics) (Theorems \ref{theorem:differomorphism-poincare-recurrence}, \ref{theorem:sigmoid-poincare-recurrence-behavior}). Due to the non-convexity of the operators $F,G$, the system can actually sometimes get stuck at equilibria, however, these fixed points may be merely artifacts of the nonlinearities of $F,G$ instead of meaningful solutions to the underline minmax problem $U$. (Theorem \ref{theorem:spurious}). 

In Section \ref{section:discrete}, we show that moving from continuous to discrete time, only enhances the disequilibrium properties of the dynamics.  Specifically, instead of energy conservation now energy increases over time leading away from equilibrium (Theorem \ref{theorem:discrete-hamiltonian}), whilst spurious (non-minmax) equilibria are still an issue 
(Theorem \ref{theorem:spurious-discrete}). Despite these negative results, there are some positive news, as at least in some cases we can show that time-averaging over these non-equilibrium trajectories (or equivalently choosing a distribution of parameters instead of a single set of parameters) can recover the min-max equilibrium (Theorem \ref{theorem:average}).  Technically our results combine tools from dynamical systems (e.g. Poincar\'{e} recurrence theorem, Poincar\'{e}-Bendixson theorem, Liouville's theorem) along with tools from game theory and non-convex optimization. 
   
Understanding the intricacies of GAN training requires broadening our vocabulary and horizons in terms of what type of long term behaviors are possible and developing new techniques that can hopefully counter them.

The structure of the rest of the paper is as follows. In Section \ref{section:related} we will present key results from prior work on the problem of min-max optimization. In Section \ref{section:preliminaries} we will present the main mathematical tools for our analysis. Sections \ref{section:two-by-two} through \ref{section:spurious} will be devoted to studying interesting special cases of hidden bilinear games. Section \ref{section:conclusion} will be the conclusion of our work.      \begin{figure}[h!]
        \centering
        \includegraphics[width=0.7\textwidth]{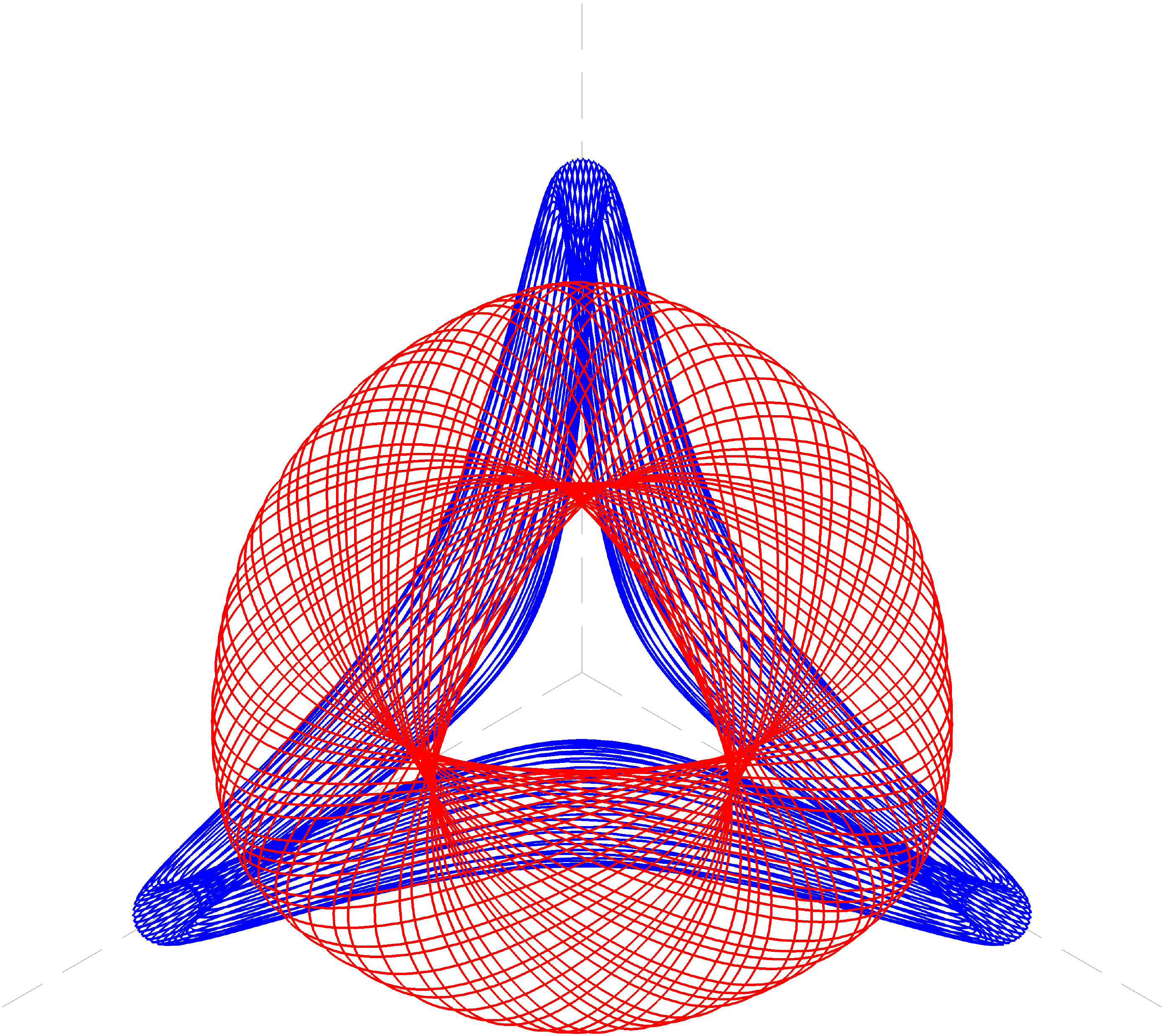}
        \caption{Trajectories of a single player using gradient-descent-ascent dynamics for a hidden Rock-Paper-Scissors game with sigmoid activations. The different colors correspond to different initializations of the dynamics. The trajectories exhibit Poincar\'{e} recurrence as expected by Theorem~ \ref{theorem:sigmoid-poincare-recurrence-behavior}.}
        \label{fig:my_label}
    \end{figure}

    \section{Related Work}\label{section:related}
\textit{Non-equilibrating dynamics in game theory.}
 \cite{paperics11} established non-convergence for a continuous-time variant of Multiplicative Weights Update (MWU), known as the replicator dynamic, for a 2x2x2 game and showed that as a result the system converges to states whose social welfare dominates that of all Nash equilibria.  
  \cite{palaiopanos2017multiplicative}
  proved the existence of Li-Yorke chaos in MWU dynamics of 2x2 potential games. 
 From the perspective of evolutionary game theory, which typically studies continuous time dynamics, numerous nonconvergence results are known but again typically for small games, e.g., \cite{sandholm10}. 
  \cite{piliouras2014optimization}
  shows that replicator dynamics exhibit a specific type of near periodic behavior in bilinear (network) zero-sum games, which is known as Poincar\'{e} recurrence.  
  Recently, \cite{mertikopoulos2018cycles}  generalized these results to more general continuous time variants of FTRL dynamics (e.g. gradient-descent-ascent).  
   Cycles arise also in evolutionary team competition \cite{DBLP:journals/corr/abs-1711-06879} as well as in network competition \cite{nagarajan2018three}.
   Technically, \cite{DBLP:journals/corr/abs-1711-06879} is the closest paper to our own as it  studies evolutionary competition between Boolean functions, however, the dynamics in the two models are different and that paper is strictly focused on periodic systems. 
  The papers in the category of cyclic/recurrent dynamics combine  delicate arguments  such as volume preservation and the existence of  constants of motions (``energy preservation"). In this paper we provide a wide generalization of these type of results by establishing cycles and recurrence type of behavior for a large class of non-convex non-concave games. 
   In the case of discrete time dynamics, such as standard gradient-descent-ascent, the system trajectories are first order approximations of the above motion  and these conservation arguments do not hold exactly. Instead, even in bilinear games, the ``energy" slowly increases over time \cite{BaileyEC18} implying chaotic divergence away from equilibrium \cite{cheung2019vortices}. We extend such energy increase results to non-linear settings.

\textit{Learning in zero-sum games and connections to GANs.}
Several recent papers have shown positive results about convergence to equilibria in (mostly bilinear) zero-sum games for suitable adapted variants of first-order methods and then apply these techniques to Generative Adversarial Networks (GANs) showing improved performance (e.g. \cite{daskalakis2017training,daskalakis2018last}).  \cite{balduzzi2018mechanics} made use of conservation laws of learning dynamics in zero-sum games (e.g. \cite{BaileyAAMAS19}) to develop new algorithms for training GANs that add a new component to the vector field that aims at minimizing this energy function. Different energy shrinking techniques for convergence in GANs (non-convex saddle point problems) exploit connections to variational inequalities and employ mirror descent techniques with an extra gradient step \cite{Gidel,mertikopoulos2018mirror}. Moreover, adding negative momentum can help with stability in zero-sum games \cite{Momentum}.  Game theoretic inspired methods such as time-averaging  work well in practice for a wide range of architectures \cite{yazici2018unusual}.  

     \section{Preliminaries}\label{section:preliminaries}
\subsection{Notation}
Vectors are denoted in boldface $\pmb{x}, \pmb{y}$ unless otherwise indicated are considered as column vectors. We use $\norm{\cdot}$ corresponds to denote the $\ell_2-$norm. For a function $f:\mathbb{R}^d\to \mathbb{R}$ we use $\nabla f$ to denote its gradient. For functions of two vector arguments, $f (\pmb{x},\pmb{y}) : \mathbb{R}^{d_1}\times \mathbb{R}^{d_2} \to \mathbb{R}$ , we use $\nabla_{\pmb{x}} f, \nabla_{\pmb{y}} f$ to denote its partial gradient. For the time derivative we will use the dot accent abbreviation, i.e.,  $\dot{\pmb{x}}=\frac{d}{dt}[\pmb{x}(t)]$. A function $f$ will belong to $C^r$ if it is $r$ times continuously differentiable. The term ``sigmoid" function refers to $\sigma:\mathbb{R}\to\mathbb{R}$ such that $\sigma(x)=(1+e^{-x})^{-1}$. Finally, we use $\measure{\cdot}$, operating over a set, to denote its (Lebesgue) measure.

\subsection{Definitions}
\begin{definition}[Hidden Bilinear Zero-Sum Game]
 In a hidden bilinear zero-sum game there are two players, each one equipped with a smooth function $\pmb{F}: \mathbb{R}^{n} \to \mathbb{R}^N$ and $\pmb{G}: \mathbb{R}^{m} \to \mathbb{R}^{M}$  and a payoff matrix $U_{N  \times M}$ such that each player inputs its own decision vector $\pmb{\theta} \in \mathbb{R}^{n}$ and $\pmb{\phi}\in \mathbb{R}^{m}$ and is trying to maximize or minimize $r(\pmb{\theta},\pmb{\phi})=\pmb{F}(\pmb{\theta})^\top U \pmb{G}(\pmb{\phi})$ respectively.
\end{definition}
In this work we will mostly study continuous time dynamics of solutions for the problem of Equation \ref{eq:min-max} for hidden bilinear zero-sum games but we will also make some important connections to discrete time dynamics that are also prevalent in practice. In order to make this distinction clear, let us  define the following terms.
\begin{definition}[Continuous Time Dynamical System]
A system of ordinary differential equations $\dot{\pmb{x}} = f (\pmb{x})$ where $f: \mathbb{R}^d \to \mathbb{R}^d$ will be called a continuous time dynamical system. Solutions of the equation $f(\pmb{x}) = 0$ are called the fixed points of the dynamical system. 
\end{definition}
We will call $f$ the \emph{vector field} of the dynamical system. In order to understand the properties of continuous time dynamical systems, we will often need to study their behaviour given different initial conditions. This behaviour is captured by the flow of the dynamical system. More precisely, 
\begin{definition}
If $f$ is Lipschitz-continuous, there exists a continuous map $\Phi(\pmb{x}_0,t)~:~\mathbb{R}^d~\times~\mathbb{R}\to~\mathbb{R}^d$ called \emph{flow} of the dynamical system such that  for all $\pmb{x}_0\in \mathbb{R}^d$
we have that $\Phi(\pmb{x}_0, t )$ is the unique solution  of the problem $\{ \dot{\pmb{x}} = f (\pmb{x}) , \pmb{x}(0) = \pmb{x}_0\}$.
We will refer to $\Phi(\pmb{x}_0, t )$ as a \emph{trajectory} or \emph{orbit} of the dynamical system.
\end{definition}

In this work we will be mainly study the gradient-descent-ascent dynamics for the problem of Equation \ref{eq:min-max}. The continuous (discrete) time version of the dynamics (with learning rate $\alpha$) are based on the following equations:
\[(\text{CGDA}):\begin{Bmatrix}
     \dot{\pmb{\theta}} &= - \nabla_{\pmb{\theta}} r(\pmb{\theta}, \pmb{\phi})\\
    \dot{\pmb{\phi}} &= \nabla_{\pmb{\phi}} r(\pmb{\theta}, \pmb{\phi})
\end{Bmatrix}\quad
\text{(DGDA}):\begin{Bmatrix}
     \pmb{\theta}_{k+1} &= \pmb{\theta}_k - \alpha\nabla_{\pmb{\theta}} r(\pmb{\theta}_k, \pmb{\phi}_k)\\
    \pmb{\phi}_{k+1} &= \pmb{\phi}_k + \alpha \nabla_{\pmb{\phi}} r(\pmb{\theta}_k, \pmb{\phi}_k)
\end{Bmatrix}
\]
A key notion in our analysis is that of (Poincaré) recurrence. Intuitively, a dynamical system is recurrent if, after a sufficiently long (but finite) time, almost every state returns arbitrarily close to the system’s initial state.
\begin{definition}\label{definition:reccurent}
A point $\mathbf{x} \in \mathbb{R}^d$ is said to be recurrent under the flow $\Phi$, if for every neighborhood $U~\subseteq~\mathbb{R}^d$ of $\mathbf{x}$, there exists an increasing sequence of times $t_n$ such that $\displaystyle\lim_{n\to\infty} t_n = \infty$ and $~\Phi(\mathbf{x},t_n)~\in~ U $ for all $n$. Moreover, the flow $\Phi$ is called Poincar\'{e} recurrent in non-zero measure set $A\subseteq \mathbb{R}^d$ if the set of the non-recurrent points in $A$ has zero measure.
\end{definition}

     \section{Cycles in hidden bilinear games with two strategies}\label{section:two-by-two}
In this section we will focus on a particular case of hidden biinear games where both the generator and the discriminator play only two strategies. Let $U$ be our zero-sum game
and without loss of generality we can assume that there are functions $f: \mathbb{R}^{n} \to [0,1]$ and $g: \mathbb{R}^{m} \to [0,1]$ such that
\begin{align*}
    \pmb{F}(\pmb{\theta}) = 
    \begin{pmatrix}
    f(\pmb{\theta}) \\
    1-f(\pmb{\theta}) 
    \end{pmatrix} 
&& 
    U = \begin{pmatrix}
            u_{0,0} && u_{0,1} \\
            u_{1,0} && u_{1,1} \\
        \end{pmatrix}
        &&
    \pmb{G}(\pmb{\phi}) = 
    \begin{pmatrix}
    g(\pmb{\phi}) \\
    1-g(\pmb{\phi}) 
    \end{pmatrix}
\end{align*} 
Let us assume that the hidden bi-linear game has a unique mixed Nash equilibrium $(p,q)$:
\begin{align*}
  v= u_{0,0} - u_{0,1} - u_{1,0} + u_{1,1}  \neq 0, \quad  q=  -\frac{u_{0,1} - u_{1,1}}{v}  \in (0,1), \quad p= -\frac{u_{1,0} - u_{1,1}}{v}   \in (0,1)
\end{align*}
Then we can write down the equations of gradient-descent-ascent : $\begin{Bmatrix}
\begin{aligned} 
    \dot{\pmb{\theta}} &= - v \nabla f(\pmb{\theta}) (g(\pmb{\phi})- q)\\
    \dot{\pmb{\phi}} &=  v \nabla g(\pmb{\phi}) (f(\pmb{\theta})- p)
\end{aligned}
\end{Bmatrix} ~\refstepcounter{equation}(\theequation)  \label{eq:eq_gda}$

In order to analyze the behavior of this system, we would like to understand the topology of the trajectories of $\pmb{\theta}$ and $\pmb{\phi}$, at least individually. The following lemma makes a connection between the trajectories of each variable in the min-max optimization system of Equation \ref{eq:eq_gda} and simple gradient ascent dynamics.
\begin{lemma}\label{lemma:reparametrization}
Let $k: \mathbb{R}^d \to \mathbb{R}$ be a $C^2$ function. Let $h: \mathbb{R} \to \mathbb{R}$ be a $C^1$ function and $\pmb{x}(t) = \rho(t)$ be the unique solution of the dynamical system $\Sigma_1$. Then for the dynamical system $\Sigma_2$ the unique solution is $\pmb{z}(t) = \rho ( \int_{0}^t h(s) \mathrm{d}s )$
\begin{equation*}
\begin{Bmatrix}
    \dot{\pmb{x}} &=& \nabla k(\pmb{x}) \\
    \pmb{x}(0) &=& \pmb{x}_0
\end{Bmatrix}: \Sigma_1\quad
\begin{Bmatrix}
    \dot{\pmb{z}} &=& h(t) \nabla k(\pmb{z}) \\
    \pmb{z}(0) &=& \pmb{x}_0
\end{Bmatrix} : \Sigma_2
\end{equation*}
\end{lemma}
By applying the previous result for $\pmb{\theta}$ with $k=f$ and $h(t) = - v (g(\pmb{\phi}(t))- q)$, we get that even under the dynamics of Equation \ref{eq:eq_gda}, $\pmb{\theta}$ remains on a trajectory of the simple gradient ascent dynamics with initial condition $\pmb{\theta}(0)$. This necessarily affects the possible values of $f$ and $g$ given the initial conditions. Let us define the sets of values attainable for each initialization.

\begin{definition}\label{definition:range}
For each  $\pmb{\theta}(0)$, $f_{\pmb{\theta}(0)}$ is the set of possible values of $f(\pmb{\theta}(t))$ can attain under gradient ascent dynamics. Similarly, we define $g_{\pmb{\phi}(0)}$ the corresponding set for $g$.
\end{definition}

What is special about the trajectories of gradient ascent is that along this curve $f$ is strictly increasing (For a detailed explanation, reader could check the proof of Theorem \ref{theorem:reparametrization} in the Appendix) and therefore each point $\pmb{\theta}(t)$ in the trajectory has a unique value for $f$. Therefore even in the system of Equation \ref{eq:eq_gda}, $f(\pmb{\theta}(t))$ uniquely identifies $\pmb{\theta}(t)$. This can be formalized in the next theorem.
\begin{theorem} \label{theorem:reparametrization}
For each $\pmb{\theta}(0),\pmb{\phi}(0)$, under the dynamics of Equation \ref{eq:eq_gda}, there are $C^1$ functions $(X_{\pmb{\theta}(0)},X_{\pmb{\phi}(0)})$ such that $X_{\pmb{\theta}(0)}:f_{\pmb{\theta}(0)}\to \mathbb{R}^{n}$
,$X_{\pmb{\phi}(0)}:g_{\pmb{\phi}(0)}\to\mathbb{R}^{n}$
and $\pmb{\theta}(t) = X_{\pmb{\theta}(0)} (f(t))$,
$\pmb{\phi}(t)~=~X_{\pmb{\phi}(0)} (g(t)).$
\end{theorem}
Equipped with these results, we are able to reduce this complicated dynamical system of $\pmb{\theta}$ and $\pmb{\phi}$ to a planar dynamical system involving $f$ and $g$ alone.
\begin{lemma}\label{lemma:functional-dynamics}
If $\pmb{\theta}(t)$ and $\pmb{\phi}(t)$ are solutions to Equation \ref{eq:eq_gda} with initial conditions $(\pmb{\theta}(0), \pmb{\phi}(0))$, then we have that $f(t) = f(\pmb{\theta}(t))$ and $g(t) = g(\pmb{\phi}(t))$ satisfy the following equations
\begin{equation} \label{eq:planar}
 \begin{aligned} 
    &\dot{f} = -v \norm{\nabla f(X_{\pmb{\theta}(0)} (f))}^2 (g- q)\\
    &\dot{g} = v \norm{\nabla g(X_{\pmb{\phi}(0)} (g))}^2 (f- p)
\end{aligned}  
\end{equation}
\end{lemma}
As one can observe both form Equation \ref{eq:eq_gda} and Equation \ref{eq:planar}, fixed points of the gradient-descent-ascent dynamics correspond to either solutions of $f(\pmb{\theta}) = p$ and $g(\pmb{\phi})=q$ or stationary points of $f$ and $g$ or even some combinations of the aforementioned conditions. Although, all of them are fixed points of the dynamical system, only the former equilibria are game theoretically meaningful. We will therefore define a subset of initial conditions for Equation \ref{eq:eq_gda} such that convergence to game theoretically meaningful fixed points may actually be feasible:
\begin{definition}\label{definition:safe-conditions}
We will call the initialization $(\pmb{\theta}(0),\pmb{\phi}(0))$ safe for Equation \ref{eq:eq_gda} if $\pmb{\theta}(0)$ and $\pmb{\phi}(0)$ are not stationary points of $f$ and $g$ respectively and $p \in f_{\pmb{\theta}(0)}$ and  $q \in g_{\pmb{\phi}(0)}$.
\end{definition}
For safe initial conditions we can show that gradient-descent-ascent dynamics applied in the class of the hidden bilinear zero-sum game mimic properties and behaviors of conservative/Hamiltonian physical systems \cite{BaileyAAMAS19}, like an ideal pendulum or an ideal spring-mass system. In such  systems, there is a notion of energy that remains constant over time and hence the system trajectories lie on level sets of these functions. To motivate further this intuition, it is easy to check that for the simplified case where $\norm{\nabla{f}}=\norm{\nabla{g}}=1$ the level sets correspond to cycles centered at the Nash equilibrium and the system as a whole captures gradient-descent-ascent for a bilinear $2\times2$ zero-sum game (e.g. Matching Pennies).
\begin{theorem} \label{theorem:invariant}
Let $\pmb{\theta}(0)$ and $\pmb{\phi}(0)$ be safe initial conditions. Then for the system of Equation \ref{eq:eq_gda}, the following quantity is time-invariant
\begin{equation*}
    H(f,g) = \int_p^f \frac{z-p}{\norm{\nabla f(X_{\pmb{\theta}(0)} (z))}^2} \mathrm{d}z + \int_q^g \frac{z-q}{\norm{\nabla g(X_{\pmb{\phi}(0)} (z))}^2} \mathrm{d}z
\end{equation*}
\end{theorem}

The existence of this invariant immediately guarantees that Nash Equilibrium $(p,q)$ cannot be reached if the dynamical system is not initialized there.
Taking advantage of the planarity of the induced system - a necessary condition of Poincar\'{e}-Bendixson Theorem -  we can prove that:

\begin{theorem} \label{theorem:orbit}
Let $\pmb{\theta}(0)$ and $\pmb{\phi}(0)$ be safe initial conditions. Then for the system of Equation \ref{eq:eq_gda}, the orbit $(\pmb{\theta}(t), \pmb{\phi}(t))$ is periodic.
\end{theorem}
On a positive note, we can prove that the time averages of $f$ and $g$ as well as the time averages of expected utilities of both players converge to their Nash equilibrium values.
\begin{theorem}\label{theorem:average}
Let $\pmb{\theta}(0)$ and $\pmb{\phi}(0)$ be safe initial conditions and $(\pmb{P},\pmb{Q})=\Big(\binom{p}{1-p},\binom{q}{1-q}\Big)$, then for the system of Equation \ref{eq:eq_gda}
\begin{equation*}
    \begin{aligned}
    \lim_{T \to \infty} \frac{\int_{0}^T   f(\pmb{\theta}(t)) \mathrm{d}t } {T} &=& p ,\quad
    \lim_{T \to \infty} \frac{\int_{0}^T  r(\pmb{\theta}(t), \pmb{\phi}(t) ) \mathrm{d}t } {T} &=& \pmb{P}^\top U \pmb{Q}
    ,\quad
    \lim_{T \to \infty} \frac{\int_{0}^T   g(\pmb{\phi}(t)) \mathrm{d}t } {T} &=& q 
    \end{aligned}
\end{equation*}
\end{theorem}
     \section{Poincar\'{e} recurrence in hidden bilinear games with more strategies}\label{section:general}
In this section we will extend our results by allowing both the generator and the discriminator to play hidden bilinear games with more than two strategies. We will specifically study the case of hidden bilinear games where each coordinate of the vector valued functions $F$ and $G$ is controlled by disjoint subsets of the variables $\pmb{\theta}$ and $\pmb{\phi}$, i.e.
\begin{equation}
 \begin{aligned}
    \pmb{\theta} &= \begin{bmatrix}
           \pmb{\theta}_1 \\
           \pmb{\theta}_2 \\
           \vdots \\
           \pmb{\theta}_N
         \end{bmatrix}\quad
    \pmb{F}(\pmb{\theta}) &= \begin{bmatrix}
           f_1(\pmb{\theta}_1) \\
           f_2(\pmb{\theta}_2) \\
           \vdots \\
           f_N(\pmb{\theta}_N)
         \end{bmatrix}\quad
    \pmb{\phi} &= \begin{bmatrix}
           \pmb{\phi}_1 \\
           \pmb{\phi}_2 \\
           \vdots \\
           \pmb{\phi}_M
         \end{bmatrix}\quad
    \pmb{G}(\pmb{\phi}) &= \begin{bmatrix}
           g_1(\pmb{\phi}_1) \\
           g_2(\pmb{\phi}_2) \\
           \vdots \\
           g_M(\pmb{\phi}_M)
         \end{bmatrix}
\end{aligned}   
\end{equation}
where each function $f_i$ and $g_i$ takes an appropriately sized vector and returns a non-negative number. To account for possible constraints (e.g. that probabilities of each distribution must sum to one), we will incorporate this restriction using Lagrange Multipliers. The resulting problem becomes
\begin{equation}\label{eq:saddle-point-optimization}
    \min_{ \pmb{\theta} \in \mathbb{R}^n,  \mu \in \mathbb{R} }\max_{ \pmb{\phi} \in \mathbb{R}^m, \lambda \in \mathbb{R}} \pmb{F}(\pmb{\theta})^\top U \pmb{G}(\pmb{\phi}) + \lambda \left(\sum_{i=1}^N f_i(\pmb{\theta}_i) -1\right)+ \mu \left(\sum_{i=j}^M g_j(\pmb{\phi}_j) -1\right)
\end{equation}
Writing down the equations of gradient-ascent-descent we get
\begin{equation}
   \begin{aligned} \label{eq:eq_gda_multi}
   &\dot{\pmb{\theta}_i} &=& - \nabla f_i(\pmb{\theta}_i) \left(\sum_{j=1}^M u_{i,j}g_j(\pmb{\phi}_j) + \lambda \right) &&
   \dot{\pmb{\phi}_j} &=& \nabla g_j(\pmb{\phi}_j) \left(\sum_{i=1}^N u_{i,j}f_i(\pmb{\theta}_i) + \mu \right)\\
   &\dot{\mu} &=& - \left(\sum_{j=1}^M g_j(\pmb{\phi}_j) -1\right) && \dot{\lambda} &=& \left(\sum_{i=1}^N f_i(\pmb{\theta}_i) -1\right)  
\end{aligned} 
\end{equation}
Once again we can show that along the trajectories of the system of Equation \ref{eq:eq_gda_multi}, $\pmb{\theta}_i$ can be uniquely identified by $f_i(\pmb{\theta}_i)$ given $\pmb{\theta}_i(0)$ and the same holds for the discriminator. This allows us to construct functions $X_{\pmb{\theta}_i(0)}$ and $X_{\pmb{\phi}_j(0)}$ just like in Theorem \ref{theorem:reparametrization}. We can now write down a dynamical system involving only $f_i$ and $g_j$.
\begin{lemma}\label{lemma:functional-dynamics-multi}
If $\pmb{\theta}(t)$ and $\pmb{\phi}(t)$ are solutions to Equation \ref{eq:eq_gda_multi} with initial conditions $(\pmb{\theta}(0), \pmb{\phi}(0),\lambda(0),\mu(0))$, then we have that $f_i(t) = f_i(\pmb{\theta}_i(t))$ and $g_j(t) = g_j(\pmb{\phi}_j(t))$ satisfy the following equations
\begin{equation}
\begin{aligned} \label{eq:f_and_g_mult}
    \dot{f_i} = -\norm{\nabla f_i(X_{\pmb{\theta}_i(0)} (f_i))}^2 \left(\sum_{j=1}^M u_{i,j} g_j + \lambda \right)\\
    \dot{g_j} =  \norm{\nabla g_j(X_{\pmb{\phi}_j(0)} (g_j))}^2 \left(\sum_{i=1}^N u_{i,j}f_i + \mu \right)
\end{aligned}    
\end{equation}
\end{lemma}
Similarly to the previous section, we can define a notion of safety for Equation \ref{eq:eq_gda_multi}. Let us assume that the hidden Game has a fully mixed Nash equilibrium $(\pmb{p}, \pmb{q})$. Then we can define
\begin{definition}
We will call the initialization $(\pmb{\theta}(0),\pmb{\phi}(0), \lambda(0), \mu(0))$ safe for Equation \ref{eq:eq_gda_multi} if $\pmb{\theta}_i(0)$ and $\pmb{\phi}_j(0)$ are not stationary points of $f_i$ and $g_j$ respectively and $p_i \in f_{i_{\pmb{\theta}_i(0)}}$ and $q_j \in g_{j_{\pmb{\phi}_j(0)}}$.
\end{definition}
\begin{theorem} \label{theorem:invariant_multi}
Assume that $(\pmb{\theta}(0),\pmb{\phi}(0), \lambda(0), \mu(0))$ is a safe initialization. Then there exist $\lambda_*$ and $\mu_*$ such that the following quantity is time invariant:
\begin{align*}
    H(\pmb{F}, \pmb{G}, \lambda, \mu) = &\sum_{i=1}^N \int_{p_i}^{f_i} \frac{z-p_i}{\norm{\nabla f_i(X_{\pmb{\theta}_i(0)} (z))}^2} \mathrm{d}z +\sum_{j=1}^M \int_{q_j}^{g_j} \frac{z-q_j}{\norm{\nabla g_j(X_{\pmb{\phi}_j(0)} (z))}^2} \mathrm{d}z  + \\
    &\int_{\lambda^*}^{\lambda} \left(z-\lambda^*\right) \mathrm{d}z
    + \int_{\mu^*}^{\mu} \left(z-\mu^*\right) \mathrm{d}z
\end{align*}
\end{theorem}
Given that even our reduced dynamical system has more than two state variables we cannot apply the Poincar\'{e}-Bendixson Theorem. Instead we can prove that there exists a one to one differentiable transformation of our dynamical system so that the resulting system becomes divergence free. Applying Louville's formula, the flow of the the transformed system is volume preserving. Combined with the invariant of Theorem \ref{theorem:invariant_multi}, we can prove that the variables of the transformed system remain bounded. This gives us the following guarantees
\begin{theorem}\label{theorem:differomorphism-poincare-recurrence}
Assume that $(\pmb{\theta}(0),\pmb{\phi}(0), \lambda(0), \mu(0))$ is a safe initialization. Then the trajectory under the dynamics of  Equation \ref{eq:eq_gda_multi} is diffeomoprphic to one trajectory of a Poincar\'{e} recurrent flow.
\end{theorem}
This result implies that if the corresponding trajectory of the Poincar\'{e} recurrent flow is itself recurrent, which almost all of them are, then the trajectory of the dynamics of Equation \ref{eq:eq_gda_multi} is also recurrent. This is however not enough to reason about how often any of the trajectories of the dynamics of Equation \ref{eq:eq_gda_multi} is recurrent. In order to prove that the flow of Equation \ref{eq:eq_gda_multi} is Poincar\'{e} recurrent we will make some additional assumptions
\begin{theorem}\label{theorem:sigmoid-poincare-recurrence-behavior}
Let $f_i$ and $g_j$ be sigmoid functions. Then the flow of  Equation \ref{eq:eq_gda_multi} is Poincar\'{e} recurrent. The same holds for all functions $f_i$ and $g_j$ that are one to one functions and for which all initializations are safe.
\end{theorem}

It is worth noting that for the unconstrained version of the previous min-max problem
we arrive at the same conclusions/theorems by repeating the above analysis without using the Lagrange multipliers.

     \section{Spurious equilibria}\label{section:spurious}
In the previous sections we have analyzed the behavior of safe initializations and we have proved that they lead to either periodic or recurrent trajectories. For initializations that are not safe for some equilibrium of the hidden game, game theoretically interesting fixed points are not even realizable solutions. In fact we can prove something stronger:

\begin{theorem} \label{theorem:spurious}
One can construct functions $f$ and $g$ for the system of Equation \ref{eq:eq_gda} so that for a positive measure set of initial conditions the trajectories converge to fixed points that do not correspond to equilibria of the hidden game.
\end{theorem}

The main idea behind our theorem is that we can construct functions $f$ and $g$ that have local optima that break the safety assumption. For a careful choice of the value of the local optima we can make these fixed points stable and then the Stable Manifold Theorem guarantees that a non zero measure set of points in the vicinity of the fixed point converges to it. Of course the idea of these constructions can be extended to our analysis of hidden games with more strategies.     \section{Discrete Time Gradient-Ascent-Descent}\label{section:discrete}
In this section we will discuss the implications of our analysis of continuous time gradient-ascent-descent dynamics on the properties of their discrete time counterparts. In general, the behavior of discrete time dynamical systems can be significantly different \cite{liyorke,BaileyEC18,palaiopanos2017multiplicative} so it is critical to perform this non-trivial analysis. We are able to show that the picture of non-equilibriation persists for an interesting class of hidden bilinear games.
\begin{theorem} 
Let $f_i$ and $g_j$ be  sigmoid functions. Then for the discretized version of the system of Equation \ref{eq:eq_gda_multi} and for safe intializations, function $H$ of Theorem \ref{theorem:invariant_multi} is non-decreasing. \label{theorem:discrete-hamiltonian}
\end{theorem}
An immediate consequence of the above theorem is that the discretized system cannot converge to the equlibrium $(\pmb{p},\pmb{q})$ if its not initialized there. For the case of non-safe initializations, the conclusions of 
Theorem  \ref{theorem:spurious}
persist in this case as well. 
\begin{theorem} \label{theorem:spurious-discrete}
One can choose a learning rate $\alpha$ and functions $f$ and $g$ for the discretized version of the system of Equation \ref{eq:eq_gda} so that for a positive measure set of initial conditions the trajectories converge to fixed points that do not correspond to equilibria of the hidden game.
\end{theorem}     \section{Conclusion}\label{section:conclusion}
In this work, inspired broadly by the structure of the complex competition between generators and discriminators in GANs, we defined a broad class of non-convex non-concave min max optimization games, which we call hidden bilinear zero-sum games. In this setting, we showed that gradient-descent-ascent behavior is considerably more complex than a straightforward convergence to the min-max solution that one might at first suspect. We showed that the trajectories even for the simplest but evocative 2x2 game exhibits cycles. In higher dimensional games, the induced dynamical system could exhibit even more complex behavior like Poincare recurrence. On the other hand, we explored safety conditions whose violation may result in convergence to spurious game-theoretically meaningless equilibria. Finally, we show that even for a simple but widespread family of functions like sigmoids discretizing gradient-descent-ascent can further intensify the disequilibrium phenomena resulting in divergence away from equilibrium.

As a consequence of this work numerous open problems emerge; Firstly, extending such recurrence results to more general families of functions, as well as examining possible generalizations to multi-player network zero-sum games are fascinating questions. Recently, there has been some progress in resolving cyclic behavior in simpler settings by employing different training algorithms/dynamics
(e.g., \cite{daskalakis2017training, mertikopoulos2019optimistic, Momentum}). It would be interesting to examine if these algorithms could enhance equilibration in our setting as well.  Additionally, the proposed safety conditions shows that a major source of spurious equilibria in GANs could be the bad local optima of the individual neural networks of the discriminator and the generator. Lessons learned from overparametrized neural network architectures that converge to global optima \cite{du2018} could lead to improved efficiency in training GANs. Finally, analyzing different simplification/models of GANs where provable convergence is possible could lead to  interesting comparisons as well as to the emergence of theoretically tractable hybrid models that capture both the hardness of GAN training (e.g. non-convergence, cycling, spurious equilibria, mode collapse, etc) as well as their power.    
     \section*{Acknowledgements}\label{section:acknowledgements}
Georgios Piliouras acknowledges MOE AcRF Tier 2 Grant 2016-T2-1-170, grant PIE-SGP-AI-2018-01 and NRF 2018 Fellowship NRF-NRFF2018-07. Emmanouil-Vasileios Vlatakis-Gkaragkounis was supported by NSF CCF-1563155, NSF CCF-1814873,  NSF CCF-1703925, NSF CCF-1763970. Finally this work was supported by the Onassis Foundation - Scholarship ID: F ZN 010-1/2017-2018. 

\bibliography{bibliography/references,bibliography/refer,bibliography/manolis_ref,bibliography/theorems,bibliography/refer2}
\bibliographystyle{plainnat}

\clearpage
\appendix
\vbox{\hsize\textwidth
		\linewidth\hsize
		\vskip 0.1in
		\hrule height 4pt
		\vskip 0.25in
		\centering
		{\LARGE\bf Poincar\'{e} Recurrence, Cycles and Spurious Equilibria in Gradient-Descent-Ascent for
Non-Convex Non-Concave Zero-Sum Games\\ \vspace{.1in}\large Supplementary Material}
		\vskip 0.29in
		\hrule height 1pt
		
	}
	
\section{Background in dynamical systems}

\subsection{Poincar\'{e}-Bendixson Theorem}
The Poincar\'{e}-Bendixson theorem is a powerful theorem that implies that two-dimensional systems cannot  exhibit chaos. Effectively, the limit behavior is either going to be an equilibrium, a periodic orbit, or a closed loop, punctuated by one (or more) fixed points. Formally, we have:
\begin{theorem}[Poincar\'{e}-Bendixson Theorem \cite{bendixson1901}]
Given a differentiable real dynamical system defined on an open subset of the plane, then every non-empty compact $\omega$-limit set of an orbit, which contains only finitely many fixed points, is either a fixed point, a periodic orbit, or a connected set composed of a finite number of fixed points together with homoclinic and heteroclinic orbits connecting these.
\end{theorem}

\subsection{Liouville’s formula and Poincar\'{e} recurrence}
In order to study the flows of dynamical systems in higher dimensions, one needs to understand more about the behaviour of the flow $\Phi$ both in time and space. An important property is the evolution of the volume of $\Phi$ over time:
\begin{theorem} [Liouville’s formula]
Let $\Phi$ be the flow of a dynamical system with vecor field $f$. Given any measurable set $A$, let $A(t) =\Phi(A, t)$ and its volume be $\mathrm{vol}[A(t)] = \int_{A(t)} \mathrm{d} \pmb{x}$. Then we have that
\begin{equation*}
    \frac{d \mathrm{vol}[A(t)]}{\mathrm{d}t} = \int_{A(t)} \mathrm{div}[f(\pmb{x})] \mathrm{d}\pmb{x}
\end{equation*}
\end{theorem}
An interesting class of dynamical systems are those whose vector fields have zero divergence everywhere.  Liouville’s formula trivially implies that the volume of the flow is preserved in such systems. This is an important tool for proving that a flow of a dynamical system is Poincar\'{e} recurrent.

\begin{theorem}[Poincar\'{e} Recurrence Theorem (version 1) \citep{Poincare1890}]{Let} $(X,\Sigma,\mu)$
be a finite measure space and let 
$f\colon X\to X$ be a measure-preserving transformation. Then,
for any $E\in \Sigma$, the set of those points $x$ of $E$ such that $f^n(x)\notin E$ for all $n>0$ has zero measure. That is, almost every point of $E$ returns to $E$. In fact, almost every point returns infinitely often.~Namely,
$$\measure{\{x\in E:\exists N \mbox{ such that } 
f^n(x)\notin E \mbox{ for all } n>N\}}=0.$$
\end{theorem}

\cite{Poincare1890} proved that in certain systems almost all trajectories return arbitrarily close to their initial position infinitely often. Indeed, let $f\colon X\to X$ be a measure-preserving transformation, $\{U_n:n\in \mathbb{N}\}$ be a basis of open sets for the bounded subset $X\subset \mathbb{R}^d$, and for each $n$ define $U_{n^\prime}=\{x\in U_n:\forall n\ge 1, f_n(x) \not\in U_n\}$. Notice that such basis exists since $\mathbb{R}^n$ is a second-countable Hausdorff space. From the initial theorem  we know that $\measure{U_{n^\prime}}=0$. Let $\mathcal{U}=\cup_{n\in \mathbb{N}} U_{n^\prime}$. Then $\measure{\mathcal{U}}=0$. We assert that if $x\in X\setminus \mathcal{U}$ then $x$ is recurrent. In fact, given a neighborhood $U$ of $x$, there is a basic neighborhood $U_n$ such that $\{x\}\subset U_n \subset U$, and since $x\not \in \mathcal{U}$ we have that $x\in U_n\setminus U_{n^\prime}$ which by definition of $U_{n^\prime}$ means that there exists $n\ge 1$ such that $f_n(x)\in U_{n}\subset U$. Thus $x$ is recurrent. Therefore, for the rest of the paper, we will use the following version which is common in dynamical systems nomenclature.

\begin{theorem}[Poincar\'{e} Recurrence Theorem (dynamical system version)]\label{theorem:poincar- Recurrence}\cite{poincarerecurrence}
If a flow $\Phi:\mathbb{R}^n\times \mathbb{R}\to \mathbb{R}^n$ preserves volume and has only orbits on a bounded subset $D$ of $\mathbb{R}^n$ then almost each point in $D$ is recurrent, i.e for every open neighborhood $U$ of $x$ there exists an increasing sequence of times $t_n$ such that $\displaystyle\lim_{n\to\infty} t_n = \infty$ and $ \Phi(\mathbf{x},t_n)\in U $ for all $n$.
\end{theorem}

\subsection{Additional Definitions}
\begin{definition}[Differomorphism, \cite{Perko}]
Let $U, V$ be manifolds. A map $f : U \rightarrow V$ is called a diffeomorphism if $f$ carries $U$ onto $V$ and also both $f$ and $f^{-1}$ are smooth.
\end{definition}

\begin{definition}[Topological conjugacy, \cite{Perko}]
Two flows $\Phi_t : A \to A$ and $\Psi_t : B \to B$ are conjugate if there exists a homeomorphism $g : A \to B$ such that
\begin{equation*}
    \forall \pmb{x} \in A, t \in \mathbb{R}: g(\Phi_t(\pmb{x})) = \Psi_t(g(\pmb{x}))
\end{equation*}
Furthermore, two flows $\Phi_t: A \to A$ and 
$\Psi_t: B \to B$ are diffeomorphic if there exists a diffeomorphism $g : A \to B $ such that 
\begin{equation*}
    \forall \pmb{x} \in A, t \in \mathbb{R}:g(\Phi_t(\pmb{x})) = \Psi_t(g(\pmb{x})).
\end{equation*}
If two flows are diffeomorphic, then their vector fields are related by the derivative of the conjugacy. That is, we get precisely the same result
that we would have obtained if we simply transformed the coordinates in their differential equations
\end{definition}

\begin{definition}[$(\alpha,\omega)$-limit set, \cite{Perko}]
Let $\Phi(\pmb{x}_0, \cdot)$ be the flow of an autonomous dynamical system $\dot{pmb{x}}=f(\pmb{x})$. Then 
\[
\begin{aligned}
\omega(\pmb{x_0}) &=& \{\pmb{x}\text{ : for all $T$ and all $\epsilon > 0$ there exists $t > T$ such that $|\Phi(\pmb{x_0}, t) - \pmb{x}| < \epsilon\}$}\\
\alpha(\pmb{x_0}) &=& \{\pmb{x}\text{ : for all $T$ and all $\epsilon > 0$ there exists $t < T$ such that $|\Phi(\pmb{x_0}, t) - \pmb{x}| < \epsilon\}$}
\end{aligned}
\]
Equivalently,
\[
\begin{aligned}
\omega(\pmb{x_0}) &=& \{\pmb{x}\text{ : there exists an unbounded, increasing sequence $\{t_k\}$ such that $\displaystyle\lim_{k\to\infty}\Phi(t_k, \pmb{x_0})= \pmb{x}$}\}\\
\alpha(\pmb{x_0}) &=& \{\pmb{x}\text{ : there exists an unbounded, decreasing sequence $\{t_k\}$ such that $\displaystyle\lim_{k\to\infty}\Phi(t_k, \pmb{x_0})= \pmb{x}$}\}\\
\end{aligned}
\]
\end{definition}

\begin{lemma}[Recurrence and Conjugacy \cite{mertikopoulos2018cycles}]\label{lem:congugacy-recurrence}
Let $\Phi_t : A \to A$ and $\Psi_t : B \to B$ be conjugate flows and $\gamma$ be the diffeomorphism which connects them. Then a point $\pmb{x} \in V$ is recurrent for $\Phi$ if and only if $\gamma(\pmb{x}) \in \gamma(V)$ is recurrent for $\Psi$.
\end{lemma}
\begin{proof}
We will first prove the if direction. Let's take any open neighborhood $U \subseteq V$ around $\pmb{x}$. Using the diffeomorphism, there is a unique $\gamma(U)  \subseteq \gamma(V)$ and additionally since $U$ is open $\gamma(U)$ is also open. Obviously, $\gamma(\pmb{x}) \in \gamma(U)$. Thus, if $\gamma(\pmb{x})$ is recurrent  there is an unbounded increasing sequence of moments $t_n$ such that
\begin{equation*}
    \Psi(\gamma(\pmb{x}),t_n)\in \gamma(U).
\end{equation*}
This is equivalent with the fact that there is an unbounded increasing sequence of moments $t_n$ such that
\begin{equation*}
    \gamma^{-1}(\Psi(\gamma(\pmb{x}),t_n))\in \gamma^{-1}(\gamma(U)).
\end{equation*}
Using the basic property of topological conjugacy, we have that 
\begin{equation*}
   \Phi(\pmb{x},t_n)=\gamma^{-1}(\Psi(\gamma(U),t_n)).
\end{equation*}
Thus, for $t_n$ we have that
\begin{equation*}
    \Phi(\pmb{x},t_n) \in U.
\end{equation*}
It follows that $\pmb{x}$ is also recurrent for $\Phi$. The result for the opposite direction follows immediately by using the inverse map.
\end{proof}

\subsection{Stable Manifold Theorems}

\begin{theorem}[Stable Manifold Theorem for Continuous Time Dynamical Systems p.120 \cite{Perko}]\label{theorem:smt-contiuous}
Let $E$ be an open subset of $\mathbb{R}^n$ containing the origin, let $f\in C^1(E)$, and let $\phi_t$ be the flow of the nonlinear system $\dot{\pmb{x}}=f(\pmb{x})$. Suppose that $f(\mathbf{0}) = \mathbf{0}$ and that $Df(\mathbf{O})$ has $k$ eigenvalues
with negative real part and $n - k$ eigenvalues with positive real part. Then
there exists a $k$-dimensional differentiable manifold $S$ tangent to the stable
subspace $E^s$ of the linear system $\dot{\pmb{x}}=Df(\mathbf{0})\pmb{x}$ at $\mathbf{0}$ such that for all $t \ge 0$, $\phi_t(S)\subseteq S$
and for all $\pmb{x}_0\in S$:
\begin{equation*}
  \displaystyle\lim_{t\to \infty}\phi_t(\pmb{x}_0)=\mathbf{0}  
\end{equation*}
and there exists an $n - k$ dimensional differentiable manifold $U$ tangent to
the unstable subspace $E^u$ of the linear system $\dot{\pmb{x}}=Df(\mathbf{0})\pmb{x}$ at $\mathbf{0}$ such that for all $t \le 0$, $\phi_t(U)\subseteq U$
and for all $\pmb{x}_0\in U$:
\begin{equation*}
  \displaystyle \lim_{t\to -\infty}\phi_t(\pmb{x}_0)=\mathbf{0} 
\end{equation*}
\end{theorem}

\begin{theorem}[Center and Stable Manifolds, p. 65 of \cite{shub}]\label{theorem:smt-discrete}
Let $\pmb{p}$ be a fixed point for the $C^r$ local diffeomorphism $h: U \to \mathbb{R}^n$ where $U \subset \mathbb{R}^n$ is an open neighborhood of $\pmb{p}$ in $\mathbb{R}^n$ and $r \geq 1$. Let $E^s \oplus E^c \oplus E^u$ be the invariant splitting
of $\mathbb{R}^n$ into generalized eigenspaces of $Dh(\pmb{p})$\footnote{Jacobian of $h$ evaluated at $\pmb{p}$.} corresponding to
eigenvalues of absolute value less than one, equal to one, and greater than one. To the $Dh(\pmb{p})$ invariant subspace $E^s\oplus
E^c$ there is an associated local $h$ invariant $C^r$ embedded disc $W^{sc}_{loc}$ of dimension $dim(E^s \oplus E^c)$, and ball $B$ around $\pmb{p}$ such that:
\begin{equation*} h(W^{sc}_{loc}) \cap B \subset W^{sc}_{loc}.\textrm{  If } h^n(\pmb{x}) \in B \textrm{ for all }n \geq 0,
\textrm{ then }\pmb{x} \in W^{sc}_{loc}.
\end{equation*}
\end{theorem}

\subsection{Regular Value Theorem}
\begin{definition}
Let $f : U \rightarrow V$ be a smooth map between same dimensional manifolds. We denote that $x \in U$ is a regular point if the derivative is nonsingular.
$y \in V$ is called a {\bf regular value} if $f^{-1}(y)$ contains only regular points.
If the derivative is singular, then $x$ is called a {\bf critical point}. We also say $y \in V$ is a critical value if $y$ is not a regular value.
\end{definition}
\begin{theorem}[Regular Value Theorem]
 If $y \in Y$ is a regular value of $f : X \rightarrow Y$ then $f^{-1}(y)$ is a manifold of dimension $n-m$, since $dim(X) = n$ and 
 $dim(Y ) = m$.
\end{theorem}

\clearpage

\section[Omitted Proofs of Section \ref{section:two-by-two}: Cycles in hidden bilinear games with two strategies]{Omitted Proofs of Section \ref{section:two-by-two}\\ Warm up: Cycles in hidden bilinear games with two strategies}
\shadowbox{
\begin{minipage}[c]{5in}
In this first section, we show a key technical lemma which will be used in many different parts of our proof. More specifically, it shows how someone can derive the solution for a non-autonomous system via a conjugate autonomous dynamical system. The main intuition is that if the non-autonomous term is multiplicative and common across all terms of a vector field then it dictates the magnitude of the vector field (the speed of
the motion), but does not affect directionality other than moving backwards or forwards along the same trajectory. 
\end{minipage}
}
\begin{lemma}[Restated \lemref{lemma:reparametrization}]\label{restated-lemma:reparametrization}
Let $k: \mathbb{R}^d \to \mathbb{R}$ be a $C^2$ function. Let $h: \mathbb{R} \to \mathbb{R}$ be a $C^1$ function and $\pmb{x}(t) = \rho(t)$ be the unique solution of the dynamical system $\Sigma_1$. Then for the dynamical system $\Sigma_2$ the unique solution is $\pmb{z}(t) = \rho ( \int_{0}^t h(s) \mathrm{d}s )$
\begin{equation*}
\begin{Bmatrix}
    \dot{\pmb{x}} &=& \nabla k(\pmb{x}) \\
    \pmb{x}(0) &=& \pmb{x}_0
\end{Bmatrix}: \Sigma_1\quad
\begin{Bmatrix}
    \dot{\pmb{z}} &=& h(t) \nabla k(\pmb{z}) \\
    \pmb{z}(0) &=& \pmb{x}_0
\end{Bmatrix} : \Sigma_2
\end{equation*}
\end{lemma}
\begin{proof}
Firstly, notice that it holds $\rho(0)= \pmb{x}_0$ and 
$ \dot{\rho} = \nabla k(\rho) $, since $\rho$ is the unique solution of $\Sigma_1$
It is easy to check that:
\begin{align*}
    \pmb{z}(0)&=\rho ( \int_{0}^0 h(s) \mathrm{d}s )=\rho(0)=\pmb{x}_0\\
    \dot{\pmb{z}}&= \nabla{\rho ( \displaystyle\int_0^t   h(s) \mathrm{d}s )}\times  \derivativeclosed{ \displaystyle\int_0^t  h(s) \mathrm{d}s}\\
    &=\nabla\rho ( \displaystyle\int_0^t   h(s) \mathrm{d}s )  h(t)
\end{align*}
\end{proof}

\shadowbox{
\begin{minipage}[c]{5in}
The next proposition states that initial condition $(\pmb{\theta}(0),\pmb{\phi}(0))$ as well as $\{f(t),g(t)\}_{t=0}^{\infty}$ are sufficient to derive the complete system state of Continuous GDA $(\theta_{\theta_0}(t),\phi_{\phi_0}(t))$.
The importance of the below theorem arises when someone takes into consideration periodicity and recurrence phenomena. Due to the existence of mapping $(f(t), g(t))$ to a unique $(\pmb{\theta}(t), \pmb{\phi}(t))$ given some initial condition $(\pmb{\theta}(0),\pmb{\phi}(0))$, any periodic or recurrent behavior of $(f(t), g(t))$ extends to the system trajectories.
\end{minipage}
}
\begin{theorem}[Restated \theoref{theorem:reparametrization}]\label{restated-theorem:reparametrization}
For each $\pmb{\theta}(0),\pmb{\phi}(0)$, under the dynamics of Equation \ref{eq:eq_gda}, there are $C^1$ functions $(X_{\pmb{\theta}(0)},X_{\pmb{\phi}(0)})$ such that $X_{\pmb{\theta}(0)}:f_{\pmb{\theta}(0)}\to \mathbb{R}^{n}$
,$X_{\pmb{\phi}(0)}:g_{\pmb{\phi}(0)}\to\mathbb{R}^{n}$
and $\pmb{\theta}(t) = X_{\pmb{\theta}(0)} (f(t))$,
$\pmb{\phi}(t)~=~X_{\pmb{\phi}(0)} (g(t)).$
\end{theorem}
\begin{proof}
Let us first study a simpler dynamical system $(\Sigma^*)$ with unique solution of  $\gamma_{\pmb{\theta}(0)}(t)$. 
\begin{align*}
  (\Sigma^*)\equiv
    \begin{Bmatrix}
    \dot{\pmb{\theta}} &=& \nabla f(\pmb{\theta}) \\
    \pmb{\theta}(0) &=& \pmb{\theta}_0
    \end{Bmatrix}  
\end{align*}
It is easy to observe that:
\begin{align*}
    \dot{f} = \nabla f(\pmb{\theta}) \dot{\pmb{\theta}} = \norm{\nabla f(\pmb{\theta})}^2 
\end{align*}
If $\pmb{x}_0$ is a stationary point of $f$ then the trajectory is a single point and the theorem holds trivially. If $\pmb{x}_0$ is not a stationary point of $f$, $f$ continuously increases along the trajectory of the dynamical system.  Therefore $A_{\pmb{\theta}(0)}(t)=f(\gamma_{\pmb{x}_0}(t))$ is an increasing function and therefore invertible. Let us call $A^{-1}_{\pmb{\theta}(0)}(f)$ the inverse.

Let's recall now the dynamical system of our interest ( \eqref{eq:eq_gda} )
\[\text{CGDA}:\begin{Bmatrix}
\begin{aligned} 
    \dot{\pmb{\theta}} &= - v \nabla f(\pmb{\theta}) (g(\pmb{\phi})- q)\\
    \dot{\pmb{\phi}} &=  v \nabla g(\pmb{\phi}) (f(\pmb{\theta})- p)
\end{aligned}
\end{Bmatrix}
\] and more precisely to the $\pmb{\theta}$-part of the system,i.e
\[
    (\Sigma)\equiv
    \begin{Bmatrix}
    \dot{\pmb{\theta}} &=& - v \nabla f(\pmb{\theta}) (g(\pmb{\phi})- q)\\
    \pmb{\theta}(0) &=& \pmb{\theta}_0
    \end{Bmatrix}
\]
Applying \lemref{restated-lemma:reparametrization} for the first equation with $h(t)=
- v  (g(\pmb{\phi}(t))- q)$,  we have that the solution of the dynamical system $(\Sigma)$ is \[\psi_{\pmb{\theta}(0)}(t) = \gamma_{\pmb{\theta}(0)}( \underbrace{\displaystyle\int_0^t   h(s) \mathrm{d}s}_{H(t)} ) = \gamma_{\pmb{\theta}(0)}(H(t)) \]

Thus it holds 
\begin{equation*}
  f(\psi_{\pmb{\theta}(0)}(t))=f(\gamma_{\pmb{\theta}(0)}(H(t)))=A_{\pmb{\theta}(0)}(H(t))  
\end{equation*}
or equivalently 
\begin{equation*}
    H(t)= A^{-1}_{\pmb{\theta}(0)}(f(\psi_{\pmb{\theta}(0)}(t)))
\end{equation*}

Plug in back to the definition of the solution, clearly we have that :\[\psi_{\pmb{\theta}(0)}(t) = \gamma_{\pmb{\theta}(0)}(
A^{-1}_{\pmb{\theta}(0)}(f(\psi_{\pmb{\theta}(0)}(t))))\]
Therefore for $X_{\pmb{\theta}(0)}(f) = \gamma_{\pmb{\theta}(0)}\circ A^{-1}_{\pmb{\theta}(0)}(f)$, which is $C^1$ as composition of $C^1$ functions, the theorem holds. 

We can perform the equivalent analysis for the $\pmb{\phi}(0)$ and $g$ and  prove that for each $\pmb{\phi}(0)$, under the dynamics Continuous GDA (Equation \ref{eq:eq_gda}), there is a $C^1$ function $X_{\pmb{\phi}(0)}~:~g_{\pmb{\phi}(0)}~\to~ \mathbb{R}^{n}$ such that $\pmb{\phi}(t) = X_{\pmb{\phi}(0)} (g(t))$.
\end{proof}

\shadowbox{
\begin{minipage}[c]{5in}
Notice that the domains of the aforementioned functions are in fact either singleton points or open intervals. This will be important when we study the safety of initial conditions. 
\end{minipage}
}
\begin{lemma}[Properties of $f_{\pmb{\theta}(0)}$] \label{lemma:interval}
If $\pmb{\theta}(0)$ is a stationary point of $f$, then $f_{\pmb{\theta}(0)}$ consists only of a single number. Otherwise, $f_{\pmb{\theta}(0)}$ is an open interval.
\end{lemma}
\begin{proof}
    If $\pmb{\theta}(0)$ is a fixed point then for the gradient ascent dynamics $\pmb{\theta}(t) = \pmb{\theta}(0)$ and therefore the Theorem holds trivially. On the other hand, in Theorem \ref{theorem:reparametrization} we argued that $f(\pmb{\theta}(t))$ is a continuous and strictly increasing function so it should map $(-\infty, \infty)$ to an open set and thus the theorem holds. Obviously we can prove an equivalent theorem for $g$.
\end{proof}
\clearpage
\shadowbox{
\begin{minipage}[c]{5in}
Having established the informational equivalence between the parameter and functional space, 
we are ready to derive the induced dynamics of the distribution with which two players participate into the game.
\end{minipage}
}
\begin{lemma}[Restated \lemref{lemma:functional-dynamics}]\label{restated-lemma:functional-dynamics}
If $\pmb{\theta}(t)$ and $\pmb{\phi}(t)$ are solutions to Equation \ref{eq:eq_gda} with initial conditions $(\pmb{\theta}(0), \pmb{\phi}(0))$, then we have that $f(t) = f(\pmb{\theta}(t))$ and $g(t) = g(\pmb{\phi}(t))$ satisfy the following equations
\begin{equation*}
 \begin{aligned} 
    &\dot{f} = -v \norm{\nabla f(X_{\pmb{\theta}(0)} (f))}^2 (g- q)\\
    &\dot{g} = v \norm{\nabla g(X_{\pmb{\phi}(0)} (g))}^2 (f- p)
\end{aligned}  
\end{equation*}
\end{lemma}
\begin{proof}
Applying chain rule and the definition of Continuous GDA (\eqref{eq:eq_gda})  we can see that :
\[
\begin{Bmatrix}    
    \dot{f} &=& \nabla f ( \pmb{\theta}(t) ) \dot{\pmb{\theta}} (t)\\ 
    \dot{g} &=& \nabla g ( \pmb{\phi}(t) ) \dot{\pmb{\phi}} (t)
\end{Bmatrix}
\Leftrightarrow
\begin{Bmatrix}    
    \dot{f} &=& -v  \norm{\nabla f ( \pmb{\theta} (t) ) }_2^2 \left( g( \pmb{\phi}(t) ) - q \right)\\
    \dot{g} &=& v  \norm{\nabla g ( \pmb{\phi} (t) ) }_2^2 \left( f( \pmb{\theta}(t) ) - p \right)
\end{Bmatrix}
\]
Finally using \theoref{theorem:reparametrization} we get:
\[
\begin{Bmatrix}    
    \dot{f} &=& -v \norm{\nabla f ( X_{\pmb{\theta}(0)} (f(t)) ) }_2^2 &\left( g( \pmb{\phi}(t) ) - q \right)\\
    \dot{g} &=& v \norm{\nabla g ( X_{\pmb{\phi}(0)} ( g(t) ) ) }_2^2 &\left( f( \pmb{\theta}(t) ) - p \right)
\end{Bmatrix}
\]
\end{proof}
\shadowbox{
\begin{minipage}[c]{5in}
Finally, we establish that the above 2-dimensional system that couples $f, g$ together is akin to a conservative system that preserves an energy-like function. Under the safety conditions, the proposed invariant is both well-defined and equipped with interesting properties. It is easy to check that it can play the role of a pseudometric around the Nash Equilibrium of the hidden bilinear game.
\end{minipage}
}
\begin{theorem}[Restated \theoref{theorem:invariant}] \label{restated-theorem:invariant}
Let $\pmb{\theta}(0)$ and $\pmb{\phi}(0)$ be safe initial conditions. Then for the system of Equation \ref{eq:eq_gda}, the following quantity is time-invariant
\begin{equation*}
    H(f,g) = \int_p^f \frac{z-p}{\norm{\nabla f(X_{\pmb{\theta}(0)} (z))}^2} \mathrm{d}z + \int_q^g \frac{z-q}{\norm{\nabla g(X_{\pmb{\phi}(0)} (z))}^2} \mathrm{d}z
\end{equation*}
\end{theorem}
\begin{proof}
Firstly, one should notice that since $\pmb{\theta}(0)$ and $\pmb{\phi}(0)$  are safe initial conditions, $H(f,g)$ is well defined  when $f,g$ follows the dynamics Continuous-GDA.
We will examine the derivative of the proposed invariant of motion.
\begin{align*}
\derivativeclosed{H(f(t),g(t))} &=&
\derivativeclosed{
\int_p^{f(t)} \frac{z-p}{\norm{\nabla f(X_{\pmb{\theta}(0)} (z))}^2} \mathrm{d}z}+ 
\derivativeclosed{
\int_q^{g(t)} \frac{z-q}{\norm{\nabla g(X_{\pmb{\phi}(0)} (z))}^2} \mathrm{d}z}\\
&=&
\derivativeclosed{f(t)}\times \frac{f(t)-p}{\norm{\nabla f(X_{\pmb{\theta}(0)} (f(t)))}^2}+
\derivativeclosed{g(t)}\times \frac{g(t)-q}{\norm{\nabla g(X_{\pmb{\phi}(0)} (g(t)))}^2}\\
\end{align*}
Using \theoref{restated-lemma:functional-dynamics}, we get
\begin{align*}
\derivativeclosed{H(f(t),g(t))} =
&-v \norm{\nabla f ( X_{\pmb{\theta}(0)} (f(t)) ) }_2^2 \left( g( \pmb{\phi}(t) ) - q \right)\times\frac{f(t)-p}{\norm{\nabla f(X_{\pmb{\theta}(0)} (f(t)))}^2}+\\
&v \norm{\nabla g ( X_{\pmb{\phi}(0)} ( g(t) ) ) }_2^2 \left( f( \pmb{\theta}(t) ) - p \right)\times \frac{g(t)-q}{\norm{\nabla g(X_{\pmb{\phi}(0)} (g(t)))}^2}\\
&=  -v({f(t)-p})({g(t)-q})+v({f(t)-p})({g(t)-q})= 0
\end{align*}
\end{proof}
\clearpage
\shadowbox{
\begin{minipage}[c]{5in}
Using the existence of the invariant function for the safe initial conditions, we will prove that the trajectory of the planar dynamical system stays bounded away from all possible fixed points. Therefore the limit behavior must be a cycle. We can also prove that the system does not just converge to a periodic orbit but it actually lies on the periodic trajectory from the very beginning. The key intuition that allows us to do this is that the level sets of $H$ are one-dimensional manifolds. To get convergence to a periodic orbit, one would require two orbits (the initial trajectory and the periodic orbit) to merge into the same one dimensional manifold, but this is not possible (requires that no transient part exists).
\end{minipage}
}
\begin{theorem}[Restated \theoref{theorem:orbit}]\label{restated-theorem:orbit}
Let $\pmb{\theta}(0)$ and $\pmb{\phi}(0)$ be safe initial conditions. Then for the system of Equation \ref{eq:eq_gda}, the orbit $(\pmb{\theta}(t), \pmb{\phi}(t))$ is periodic.
\end{theorem}
\begin{proof}

If $(\pmb{\theta}(0), \pmb{\phi}(0))$ is a fixed point then it is trivially a periodic point. Suppose $(\pmb{\theta}(0), \pmb{\phi}(0))$ is not a fixed point, then either $f \neq p$ or $g \neq q$ (or both). Given that $H$ is invariant, the trajectory of the planar system stays bounded away from all equilibria. We will examine each case separately:
\paragraph{Equilbria with $f=p$ and $g=q$}
It is bounded away from these since $H(p, q) = 0$ and $H(f(\pmb{\theta}(0)), g(\pmb{\phi}(0)))> 0$.
\paragraph{Equilibria with  $f = p$ and $\nabla f = \mathbf{0}$}
These equilibria are not achievable since they are not allowed by the safety conditions. $\nabla f = \mathbf{0}$ when $f=p$ means that $p$ is one of the endpoints of $f_{\pmb{\theta}(0)}$. But by Lemma \ref{lemma:interval}, $f_{\pmb{\theta}(0)}$ is an open set and $p \in f_{\pmb{\theta}(0)}$ which leads to a contradiction.
\paragraph{Equilibria with  $g = q$ and $\nabla g = \mathbf{0}$}
They are also not feasible due to the safety assumption.
\paragraph{Equilibria with $\nabla f = \mathbf{0}$ and $\nabla g = \mathbf{0}$}
Observe that such points lie in the corners of $ f_{\pmb{\theta}(0)} \times g_{\pmb{\phi}(0)}$. These points correspond to local maxima of the invariant function. We will prove this for one of the corners and the same proof works for all others in the same way. Let $(p^*, q^*)$ be one such corner with both $p^*> p$ and $q^*> q$. Let us take any other point $(r,z)$ with $p^* \geq  r >p$ and $q^* \geq z>q$ but different from $(p,q)$. Without loss of generality let us assume $p^* >  r$. Then in this region $H$ is increasing in both $f$ and $g$. Thus
\begin{equation*}
    H(r, z) < H(p^*, z) \leq H(p^*, q^*)
\end{equation*}
So this corner (and all the other three corners) are local maxima. A continuous trajectory cannot reach these isolated local maxima while maintaining $H$ invariant. 

\begin{figure}[ht!]
    \centering
    \includegraphics[width=\textwidth]{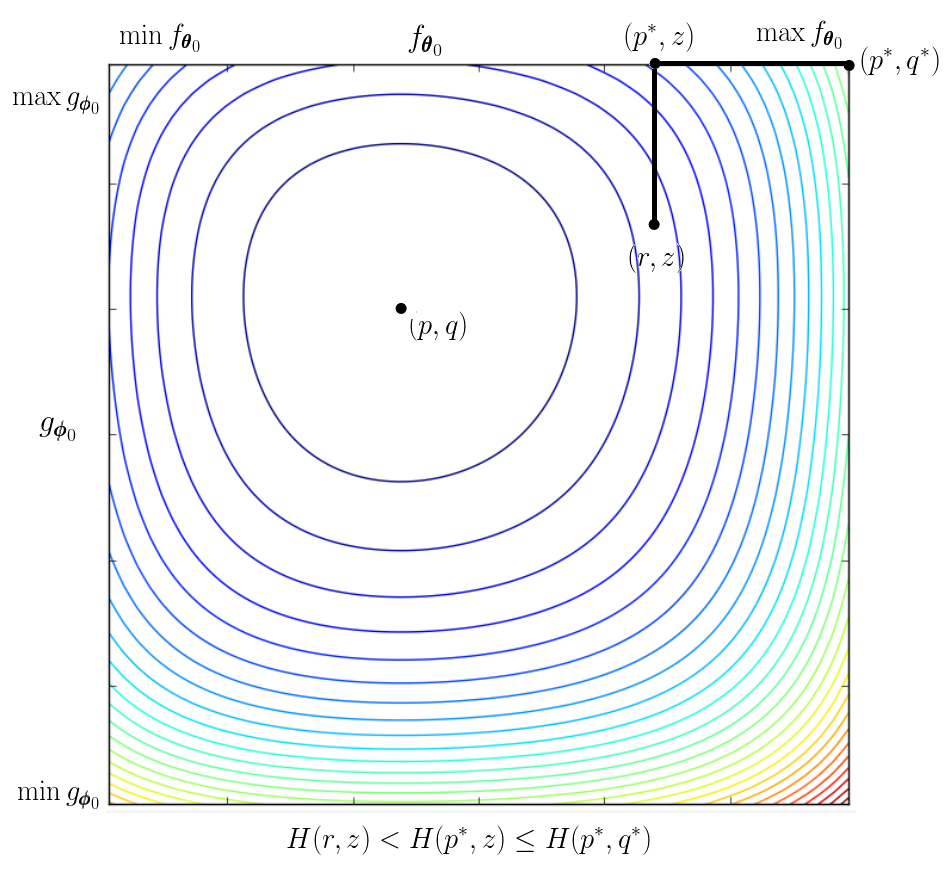}
\end{figure}

Thus we can create a trapping/invariant region $C$ so that $f$ and $g$ always stay in $C$ and $C$ does not contain any fixed points. By the Poincar\'{e}-Bendixson theorem, the $\alpha,\omega$-limit set of the trajectory is a periodic orbit. Thus they 
are isomorphic to $S^1$.

Since the gradient of $H$ is only equal to $0$ at $(p, q)$
\begin{equation*}
    \nabla H = \left( \frac{f-p}{\norm{\nabla f(X_{\pmb{\theta}(0)} (f))}^2}, \frac{g-q}{\norm{\nabla g(X_{\pmb{\phi}(0)} (g))}^2} \right) 
\end{equation*}

Therefore $H(f(\pmb{\theta}(0)), g(\pmb{\phi}(0))) > H(p,q)$ is a regular value of $H$. By the regular value theorem the following set is a one dimensional manifold 
\begin{equation*}
    \{(f,g) \in f_{\pmb{\theta}(0)} \times g_{\pmb{\phi}(0)} : H(f,g) =  H(f(\pmb{\theta}(0)), g(\pmb{\phi}(0))) \}   
\end{equation*}
Notice that by the invariance of $H$ and definition of  $\alpha,\omega-$limit sets of $(f(\pmb{\theta}(0)),g(\pmb{\phi}(0)))$, we know that both the trajectory starting at $(\pmb{\theta}(0),\pmb{\phi}(0))$, along with its $\alpha,\omega-$limit sets belong to the above manifold. Thus, their union is a closed, connected $1-$manifold and thus it is isomorphic to $S^1$.

Assume that the trajectory was merely converging to the $\alpha,\omega-$limit sets. Then our one dimensional manifold is containing two connected one dimensional manifolds: the trajectory of the system as well as the $\alpha,\omega-$limit sets . But one can easily show that this would not be a one dimensional manifold, leading to a contradiction.

\begin{figure}[ht!]
    \centering
    \includegraphics[width=\textwidth]{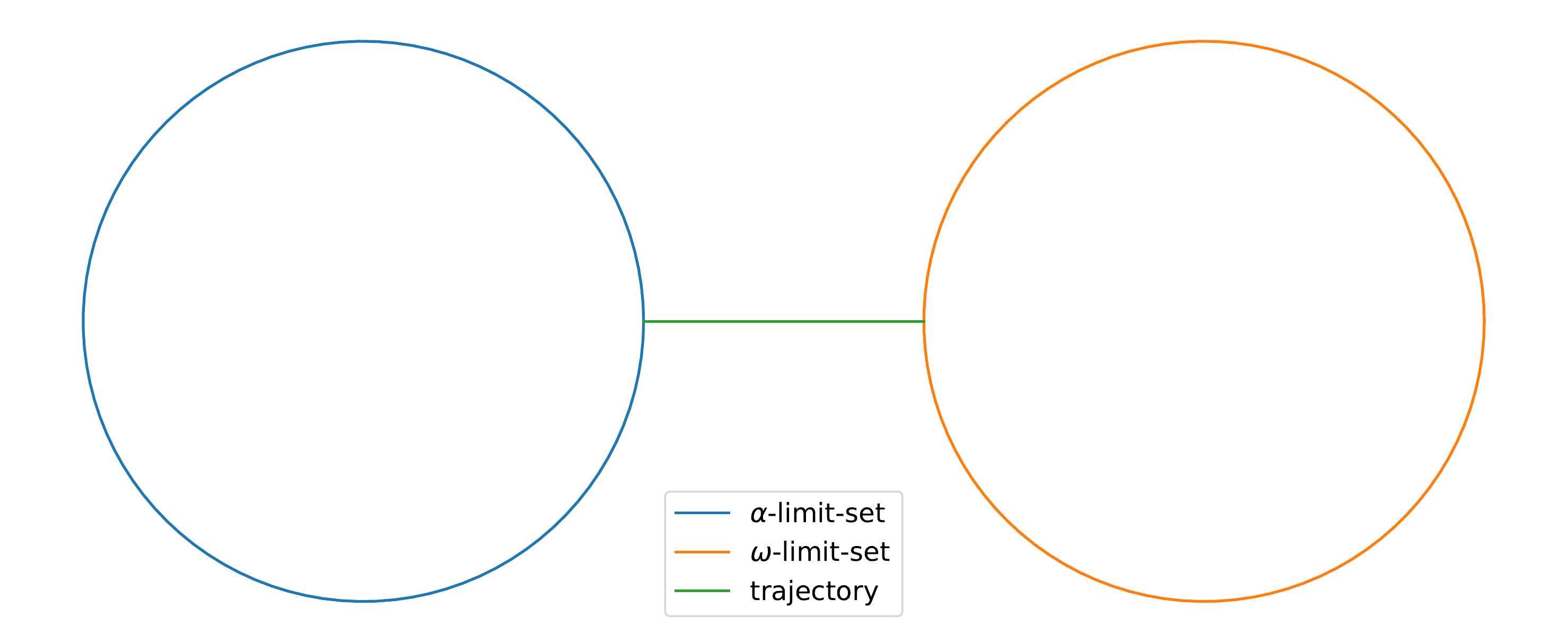}
    \caption{By the Poincar\'{e}-Bendixson theorem we know that both the $\alpha$ and the $\omega$ limit-sets are isomorphic to $S^1$. The trajectory connecting them makes the union of all three parts is not a one dimensional manifold. But by the regular value theorem on $H$, the union of all three parts is also a one dimensional manifold.}
\end{figure}

Up to now we have analyzed the trajectories of the planar dynamical system of $f$ and $g$. But since we have proved that there is one to one correspondence between $\pmb{\theta}$ and $f$ and $\pmb{\phi}$ and $g$, the periodicity claims transfer to  $\pmb{\theta}(t)$ and $\pmb{\phi}(t)$.
\end{proof}
\clearpage

\shadowbox{
\begin{minipage}[c]{5in}
On a positive note, one can prove that the time average of $f$ and $g$ do converge as well as the utilities of the generator and discriminator.
\end{minipage}
}
\begin{theorem}[Restated \theoref{theorem:average}]\label{restated-theorem:average}
Let $\pmb{\theta}(0)$ and $\pmb{\phi}(0)$ be safe initial conditions and $(\pmb{P},\pmb{Q})=\Big(\binom{p}{1-p},\binom{q}{1-q}\Big)$, then for the system of Equation \ref{eq:eq_gda}
\begin{equation*}
    \begin{aligned}
    \lim_{T \to \infty} \frac{\int_{0}^T   f(\pmb{\theta}(t)) \mathrm{d}t } {T} &=& p ,\quad
    \lim_{T \to \infty} \frac{\int_{0}^T  r(\pmb{\theta}(t), \pmb{\phi}(t) ) \mathrm{d}t } {T} &=& \pmb{P}^\top U \pmb{Q}
    ,\quad
    \lim_{T \to \infty} \frac{\int_{0}^T   g(\pmb{\phi}(t)) \mathrm{d}t } {T} &=& q 
    \end{aligned}
\end{equation*}
\end{theorem}
\begin{proof}
In Theorem \theoref{theorem:orbit} we have discussed that the safety of the initial conditions guarantees that stationary points of $f$ and $g$ are going to be avoided. So using \lemref{lemma:functional-dynamics}, we can integrate the following quantities over a time interval $[0,T]$ and divide by $T$.

\begin{align*}
 \frac{1}{T} \displaystyle\int_0^T\dfrac{1}{v \norm{\nabla f ( X_{\pmb{\theta}(0)} (f(t)) ) }^2}\derivative{f} \mathrm{d}t &= - \frac{1}{T} \int_0^T \left(g( \pmb{\phi}(t) ) - q \right) \mathrm{d}t\\
 \frac{1}{T}\displaystyle\int_0^T\dfrac{1}{v \norm{\nabla g ( X_{\pmb{\phi}(0)} ( g(t) ) ) }^2 }\derivative{g} \mathrm{d}t &= \frac{1}{T} \int_0^T \left( f( \pmb{\theta}(t) ) - p \right)   \mathrm{d}t
\end{align*}    
Let us define the follwoing functions of $f$ and $g$:
\begin{align*}
    \mathbb{F}(f(t))=v \norm{\nabla f ( X_{\pmb{\theta}(0)} (f(t)) ) }^2\\
    \mathbb{G}(g(t))=v \norm{\nabla g ( X_{\pmb{\phi}(0)} ( g(t) ) ) }^2
\end{align*}
Thus the above dynamical system is equivalent with:
\begin{align*}
    \frac{1}{T} \displaystyle\int_0^T\dfrac{1}{\mathbb{F}(f(t))}\derivative{f} \mathrm{d}t &= - \frac{1}{T}\displaystyle\displaystyle\int_0^T \left( g( \pmb{\phi}(t) ) - q \right)\mathrm{d}t\\
    \frac{1}{T} \displaystyle\int_0^T\dfrac{1}{\mathbb{G}(g(t))}\derivative{g} \mathrm{d}t &=
    \frac{1}{T} \displaystyle\displaystyle\int_0^T \left( f( \pmb{\theta}(t) ) - p \right)\mathrm{d}t\\
\end{align*}
 
 However, by a simple change of variables we have that :
\begin{align*}
 \displaystyle\int_0^T \frac{1}{\mathbb{F}(f)}\frac{df}{dt} dt=  \int_{f(0)}^{f(T)}\frac{1}{\mathbb{F}(f)}df \\
 \displaystyle\int_0^T \frac{1}{\mathbb{G}(g)}\frac{dg}{dt} dt=  \int_{g(0)}^{g(T)}\frac{1}{\mathbb{G}(g)}dg
\end{align*}
However we know that $f(t),g(t)$ for our dynamical system are periodic and bounded away from the roots of $\mathbb{F}(f),\mathbb{G}(g)$. So their integrals over a single period of $f$ and $g$ are bounded and we have that
\begin{align*}
    \displaystyle\lim_{T\rightarrow \infty} \frac{1}{T} \displaystyle\int_0^T \frac{1}{\mathbb{F}(f)}\derivative{f} dt= \displaystyle\lim_{T\rightarrow \infty} \frac{1}{T} \displaystyle\int_{f(0)}^{f(T)}\frac{1}{\mathbb{F}(f)}df =0 \\
 \displaystyle\lim_{T\rightarrow \infty} \frac{1}{T} \displaystyle\int_0^T \frac{1}{\mathbb{G}(g)}\derivative{g} dt = \displaystyle\lim_{T\rightarrow \infty} \frac{1}{T} \displaystyle\int_{g(0)}^{g(T)}\frac{1}{\mathbb{G}(g)}dg =0
\end{align*}
Therefore, 
\begin{align*}
  \displaystyle\lim_{T\rightarrow \infty}\frac{1}{T} \displaystyle\displaystyle\int_0^T \left( g( \pmb{\phi}(t)) -q )  \right)\mathrm{d}t = 0\\
  \displaystyle\lim_{T\rightarrow \infty}\frac{1}{T} \displaystyle\displaystyle\int_0^T \left( f( \pmb{\theta}(t)) -p )  \right)\mathrm{d}t = 0\\  
\end{align*}
which implies
\begin{align*}
\displaystyle\lim_{T\rightarrow \infty} \frac{\displaystyle\int_0^T g(\pmb{\phi} (t)) dt}{T}= q  \quad
\displaystyle\lim_{T\rightarrow \infty} \frac{\displaystyle\int_0^T f(\pmb{\theta} (t) ) dt}{T}= p
\end{align*}
Next, we will proceed with the argument about the time average of the objective function. 

\begin{fact*} If $(\pmb{P},\pmb{Q})$ is fully mixed Nash Equilibrium, then it holds
\begin{align*}
    &\pmb{P}^\top U \pmb{G} (\pmb{\phi}(t)) = \pmb{F}(\pmb{\theta}(t))^\top U \pmb{Q} =  \pmb{P}^\top U \pmb{Q} \\
    &(\pmb{F}(\pmb{\theta}(t))-\pmb{P})^\top  U (\pmb{G}(\pmb{\phi}(t))- \pmb{Q})={\pmb{F}(\pmb{\theta}(t))}^\top  U \pmb{G}(\pmb{\phi}(t)) - {\pmb{P}}^\top  U \pmb{Q}
\end{align*}
\end{fact*}
\begin{proof}
It suffices to prove the first part of the claim, since the second part is its immediate 
consequence. 
Since we have conditioned that $(\pmb{P},\pmb{Q})$ is a fully mixed Nash Equilibrium, it holds : 
\begin{equation*}
   \pmb{P}^\top U \pmb{Q} =  \binom{1}{0}^\top U \pmb{Q} = \binom{0}{1}^\top U \pmb{Q}  
\end{equation*}
Therefore:  
\begin{equation*}
    \pmb{F} (\pmb{\theta})^\top U \pmb{Q} = f(\pmb{\theta}) \binom{1}{0}^\top U \pmb{Q} + (1-f(\pmb{\theta})) \binom{0}{1}^\top U \pmb{Q} = \pmb{P}^\top U \pmb{Q} 
\end{equation*}
Symmetrically, it holds :
\begin{equation*}
    \pmb{P}^\top U \pmb{Q} =  \pmb{P}^\top U \binom{1}{0} = \pmb{P}^\top U \binom{0}{1}.
\end{equation*}
Therefore 
\begin{equation*}
    \pmb{P}^\top U \pmb{Q}= \pmb{P}^\top U \binom{1}{0} g(\pmb{\phi}(t)) + \pmb{P}^\top U\binom{0}{1} (1-g(\pmb{\phi}(t))) = \pmb{P}^\top U G (\pmb{\phi}(t)).
\end{equation*}
\end{proof}
Observe the following fact:
\begin{align*}
    \frac{1}{T}\int_0^T   {\pmb{F}(\pmb{\theta}(t)})^\top  U \pmb{G}(\pmb{\phi}(t))\mathrm{d}t - {\pmb{P}}^\top  U \pmb{Q} &= \frac{1}{T}\displaystyle\int_0^T   {\pmb{F}(\pmb{\theta}(t)})^\top  U \pmb{G}(\pmb{\phi}(t))\mathrm{d}t -\frac{1}{T}\int_0^T   {\pmb{P}}^\top  U \pmb{Q}\mathrm{d}t\\
    &= \frac{1}{T}\int_0^T  ({\pmb{F}(\pmb{\theta}(t))-\pmb{P}})^\top  U (\pmb{G}(\pmb{\phi}(t))- \pmb{Q}) \mathrm{d}t 
\end{align*}
Therefore it suffices to show that
\begin{equation*}
    \lim_{T\rightarrow \infty}\frac{1}{T}\int_0^T  ({\pmb{F}(\pmb{\theta}(t))-\pmb{P}})^\top  U (\pmb{G}(\pmb{\phi}(t)- \pmb{Q}) \mathrm{d}t =0
\end{equation*}
The payoff matrix $U$ is as follows:
$$
U =
 \begin{pmatrix}
  u_{0,0} & u_{1,0} \\
  u_{1,0} & u_{1,1}
 \end{pmatrix}
 $$ 
We have that 
\begin{equation*}
  ({\pmb{F}(\pmb{\theta}(t))-\pmb{P}})^\top  U (\pmb{G}(\pmb{\phi}(t))- \pmb{Q})= (u_{0,0}-u_{1,0}-u_{1,0}+u_{1,1})(f(\pmb{\theta}(t))-p))(g(\pmb{\phi}(t))-q). 
\end{equation*}
 Therefore it suffices to show that :
\begin{equation*}
    \lim_{T\rightarrow \infty}\frac{1}{T}\int_0^T  (f(\pmb{\theta}(t)) -p))(g(\pmb{\phi}(t))-q) \mathrm{d}t =0.
\end{equation*}
 
By our previous analysis in this theorem, we have already argued that 
\begin{equation*}
    \lim_{T\rightarrow \infty}\frac{1}{T}\int_0^T  (g(\pmb{\phi}(t))-q) \mathrm{d}t =0
\end{equation*}
thus we only have to show that  
\begin{equation*}
    \lim_{T\rightarrow \infty}\frac{1}{T}\int_0^T  f(\pmb{\theta}(t)) (g(\pmb{\phi}(t))-q) \mathrm{d}t =0
\end{equation*}
Revisiting the equations of Lemma \ref{lemma:functional-dynamics}:
\begin{align*}
    \frac{f}{\mathbb{F}(f)} \frac{df}{\mathrm{d}t} &= f(\pmb{\theta}(t))(g(\pmb{\phi}(t))-q) \Rightarrow \\
    \frac{1}{T} \int_0^T   \frac{f}{\mathbb{F}(f)} \frac{df}{\mathrm{d}t} \mathrm{d}t &=\frac{1}{T}\int_0^T  f(\pmb{\theta}(t)) (g(\pmb{\phi}(t))-q) \mathrm{d}t
\end{align*}
However using similar arguments as before we can prove that
\begin{equation*}
   \lim_{T\rightarrow \infty} \frac{1}{T} \int_0^T   \frac{f}{\mathbb{F}(f)} \frac{df}{\mathrm{d}t} \mathrm{d}t =  \lim_{T\rightarrow \infty}\frac{1}{T} \displaystyle\int_{f(0)}^{f(T)} \frac{f}{\mathbb{F}(f)}df=0
\end{equation*}
implying  that
\begin{equation*}
   \lim_{T\rightarrow \infty} \frac{1}{T}\int_0^T  f(\pmb{\theta}(t)) (g(\pmb{\phi}(t))-q) \mathrm{d}t = 0
\end{equation*}
which completes the proof.
\end{proof}

\clearpage
\section[Omitted Proofs of Section \ref{section:general}: Poincar\'{e} recurrence in hidden bilinear games with more strategies]{Omitted Proofs of Section \ref{section:general}\\ Poincar\'{e} recurrence in hidden bilinear games with more strategies}
\begin{lemma}[Restated \lemref{lemma:functional-dynamics-multi}]
\label{restated-lemma:functional-dynamics-multi}
If $\pmb{\theta}(t)$ and $\pmb{\phi}(t)$ are solutions to Equation \ref{eq:eq_gda_multi} with initial conditions $(\pmb{\theta}(0), \pmb{\phi}(0),\lambda(0),\mu(0))$, then we have that $f_i(t) = f_i(\pmb{\theta}_i(t))$ and $g_j(t) = g_j(\pmb{\phi}_j(t))$ satisfy the following equations
\begin{equation*}
\begin{aligned}
    \dot{f_i} = -\norm{\nabla f_i(X_{\pmb{\theta}_i(0)} (f_i))}^2 \left(\sum_{j=1}^M u_{i,j} g_j + \lambda \right)\\
    \dot{g_j} =  \norm{\nabla g_j(X_{\pmb{\phi}_j(0)} (g_j))}^2 \left(\sum_{i=1}^N u_{i,j}f_i + \mu \right)
\end{aligned}    
\end{equation*}
\end{lemma}
\begin{proof}
Applying chain rule  we can see that :
\begin{align*}
    \forall i \in [N] : \dot{f_i} &= \nabla f_i ( \pmb{\theta}_i(t) ) \dot{\pmb{\theta}_i} (t) \\
    \forall j \in [M] : \dot{g_j} &= \nabla g_j ( \pmb{\phi}_j(t) ) \dot{\pmb{\phi}_j} (t) 
\end{align*}
Then by the dynamics of Continuous GDA (\eqref{eq:eq_gda})
\begin{align*}
    \forall i \in [N] : \dot{f_i} &= \nabla f_i ( \pmb{\theta}_i(t) ) \left(
    - \nabla f_i(\pmb{\theta}_i) \left(\sum_{j=1}^M u_{i,j}g_j(\pmb{\phi}_j) + \lambda \right)
    \right) \\
    \forall j \in [M] : \dot{g_j} &= \nabla g_j ( \pmb{\phi}_j(t) ) \left(
     \nabla g_j(\pmb{\phi}_j) \left(\sum_{i=1}^N u_{i,j}f_i(\pmb{\theta}_i) + \mu \right) \right)\\
\end{align*}
Clearly
\begin{align*}
    \forall i \in [N] : \dot{f_i} &= -\norm{\nabla f_i ( \pmb{\theta}_i(t) )}^2 \left(\displaystyle\sum_{j=1}^M u_{i,j}g_j(\pmb{\phi}_j) + \lambda \right)    \\
    \forall j \in [M] : \dot{g_j} &= \norm{\nabla g_j ( \pmb{\phi}_j(t) )}^2  \left(\displaystyle\sum_{i=1}^N u_{i,j}f_i(\pmb{\theta}_i) + \mu \right)     
\end{align*}

Finally using \theoref{theorem:reparametrization} we know that there exist $N+M$ functions such that :
\begin{align*}
   \pmb{\theta}_i(t) &= X_{\pmb{\theta}_i(0)} (f_i (t))\\
\pmb{\phi}_j(t) &= X_{\pmb{\phi}_j(0)} (g_j (t))
\end{align*}  
 Combining the last two expressions we get the desired claim.
\end{proof}

\begin{theorem}[Restated \theoref{theorem:invariant_multi}]\label{restated-theorem:invariant_multi}
Assume that $(\pmb{\theta}(0),\pmb{\phi}(0), \lambda(0), \mu(0))$ is a safe initialization. Then there exist $\lambda_*$ and $\mu_*$ such that the following quantity is time invariant:
\begin{align*}
    H(\pmb{F}, \pmb{G}, \lambda, \mu) = &\sum_{i=1}^N \int_{p_i}^{f_i} \frac{z-p_i}{\norm{\nabla f_i(X_{\pmb{\theta}_i(0)} (z))}^2} \mathrm{d}z +\sum_{j=1}^M \int_{q_j}^{g_j} \frac{z-q_j}{\norm{\nabla g_j(X_{\pmb{\phi}_j(0)} (z))}^2} \mathrm{d}z  + \\
    &\int_{\lambda^*}^{\lambda} \left(z-\lambda^*\right) \mathrm{d}z
    + \int_{\mu^*}^{\mu} \left(z-\mu^*\right) \mathrm{d}z
\end{align*}
\end{theorem}
\begin{proof}
We know that $(\pmb{p}, \pmb{q})$ is an equilibrium of the hidden bilinear game
\begin{equation}\label{eq:simple-bilinear}
   \min_{ \pmb{x} \in \Delta_N}\max_{ \pmb{y} \in \Delta_M} \pmb{x}^\top U \pmb{y}
\end{equation}
Let us make the same Lagrangian transformation we did in  Section \ref{section:general}. 
\begin{equation}\label{eq:saddle-lagrange}
   \min_{ \pmb{x} \geq 0, \mu \in \mathbb{R} }\max_{ \pmb{y} \geq 0, \lambda \in \mathbb{R}} \pmb{x}^\top U \pmb{y} + \mu \left( \sum_{i=1}^M y_i\right) + \lambda \left( \sum_{j=1}^N x_j\right)
\end{equation}

Since $(\pmb{p},\pmb{q})$ is an equilibrium of the problem of Equation \ref{eq:simple-bilinear}, the KKT conditions on the Problem of Equation \ref{eq:saddle-lagrange} imply that there are (unique) $\lambda^*,\mu^*$
\begin{align*}
    \forall j\in [M]:\quad \displaystyle\sum_{i\in [N]}u_{i,j}p_i+\mu^*=0\\
    \forall i\in [N]:\quad \displaystyle\sum_{j\in [M]}u_{i,j}q_j+\lambda^*=0\\ 
\end{align*}

We will analyze the time derivative of $H(\pmb{F}(t), \pmb{G}(t), \lambda(t), \mu(t))$ over the trajectory of CGDA (\eqref{eq:eq_gda_multi}).
\begin{align*}
    H(\pmb{F}, \pmb{G}, \lambda, \mu) = &\displaystyle\sum_{i=1}^N \int_{p_i}^{f_i} \frac{z-p_i}{\norm{\nabla f_i(X_{\pmb{\theta}_i(0)} (z))}^2} \mathrm{d}z +\displaystyle\sum_{j=1}^M \int_{q_j}^{g_j} \frac{z-q_j}{\norm{\nabla g_j(X_{\pmb{\phi}_j(0)} (z))}^2} \mathrm{d}z  + \\
    &\int_{\lambda^*}^{\lambda} \left(z-\lambda^*\right) \mathrm{d}z
    + \int_{\mu^*}^{\mu} \left(z-\mu^*\right) \mathrm{d}z\Rightarrow
\end{align*}
\begin{align*}
    \derivativeclosed{H(\pmb{F}(t), \pmb{G}(t), \lambda(t), \mu(t))}&=
    \sum_{i=1}^N \dot{f_i} \frac{f_i-p_i}{\norm{\nabla f_i(X_{\pmb{\theta}_i(0)} (f_i))}^2}+\sum_{j=1}^M \dot{g_j} \frac{g_j-q_j}{\norm{\nabla g_j(X_{\pmb{\phi}_j(0)} (g_j))}^2}\\
    &+\dot{\lambda}\left(\lambda-\lambda^*\right) +\left(\mu-\mu^*\right)\dot{\mu} \\
    \derivativeclosed{H(\pmb{F}(t), \pmb{G}(t), \lambda(t), \mu(t))}&=
    \sum_{i=1}^N \left(\sum_{j=1}^M u_{i,j} g_j + \lambda \right)  (p_i-f_i)\\
    &+\sum_{j=1}^M \left(\sum_{i=1}^N u_{i,j}f_i + \mu \right))(g_j-q_j)\\
    &+\left(\lambda-\lambda^*\right)\dot{\lambda} +\left(\mu-\mu^*\right)\dot{\mu} \\
\end{align*}
Applying the KTT conditions we have
\begin{align*}
    \sum_{j=1}^M u_{i,j} g_j + \lambda  = \sum_{j=1}^M u_{i,j} (g_j - q_j) + \lambda -\lambda^*\\
    \sum_{i=1}^N u_{i,j}f_i + \mu  = \sum_{i=1}^N u_{i,j} (f_i - p_i) + \mu -\mu^*
\end{align*}
We can now write down:
\begin{align*}
    \sum_{i=1}^N\sum_{j=1}^M u_{i,j} g_j (p_i-f_i)+ \lambda  = \sum_{i=1}^N \sum_{j=1}^M u_{i,j} (g_j - q_j) (p_i-f_i) + (\lambda -\lambda^*) \sum_{i=1}^N (p_i-f_i)\\
    \sum_{j=1}^M\sum_{i=1}^N u_{i,j}f_i (g_j-q_j) + \mu  = \sum_{j=1}^M\sum_{i=1}^N u_{i,j} (f_i - p_i) (g_j - q_j) +(\mu -\mu^*)\sum_{j=1}^M (g_j-q_j)
\end{align*}
Observe that summing the two expressions the $u_{i,j}$ terms cancel out. Thus we can write
\begin{align*}
    \derivativeclosed{H(\pmb{F}(t), \pmb{G}(t), \lambda(t), \mu(t))}&=
    (\lambda -\lambda^*) \sum_{i=1}^N (p_i-f_i)  + +(\mu -\mu^*)\sum_{j=1}^M (q_j-g_j) \\
    &+\left(\lambda-\lambda^*\right)\dot{\lambda} +\left(\mu-\mu^*\right)\dot{\mu} \\
\end{align*}
Additionally we have that $\pmb{p}$ and $\pmb{q}$ are probability vectors so
\begin{align*}
    \dot{\lambda} &= \sum_{i=1}^N f_i - 1 = \sum_{i=1}^N (f_i - p_i) \\
    \dot{\mu} &= - \left(\sum_{j=1}^M g_j -1 \right) = - \sum_{j=1}^M (g_j -q_j)
\end{align*}
Thus
\begin{align*}
    \derivativeclosed{H(\pmb{F}(t), \pmb{G}(t), \lambda(t), \mu(t))}=0
\end{align*}
\end{proof}

\shadowbox{
\begin{minipage}[c]{5in}
Since the proof of the following Theorem is fairly complicated, we will firstly outline the basic steps below:
\begin{enumerate}
    \item We first show that there is topological conjugate dynamical system
    whose dynamics are \emph{incompressible}  i.e. the volume of a set of initial conditions remains invariant as the dynamics evolve over time. By \theoref{theorem:poincar- Recurrence}, if every solution remains in a bounded space for all $t\geq0$, incompressibility implies recurrence.
    \item To establish boundedness in these dynamics, we exploit the aforementioned invariant function.
\end{enumerate}
\end{minipage}
}

\begin{theorem}[Restated \theoref{theorem:differomorphism-poincare-recurrence}]\label{restated-theorem:differomorphism-poincare-recurrence}
Assume that $(\pmb{\theta}(0),\pmb{\phi}(0), \lambda(0), \mu(0))$ is a safe initialization. Then the trajectory under the dynamics of  Equation \ref{eq:eq_gda_multi} is diffeomoprphic to one trajectory of a Poincar\'{e} recurrent flow.
\end{theorem}
\begin{proof}
Let us start with the dynamics of Equation \ref{eq:eq_gda_multi}. We we call its flow $\Phi_{\textrm{original}}$:
\begin{equation*}
   \Sigma_{\textrm{original}}:
   \begin{Bmatrix}
   \dot{\pmb{\theta}_i} &=& - \nabla f_i(\pmb{\theta}_i) \left(\displaystyle\sum_{j=1}^M u_{i,j}g_j(\pmb{\phi}_j) + \lambda \right) & 
   \dot{\pmb{\phi}_j} &=& \nabla g_j(\pmb{\phi}_j) \left(\displaystyle\sum_{i=1}^N u_{i,j}f_i(\pmb{\theta}_i) + \mu \right)\\
   \dot{\mu} &=& - \left(\displaystyle\sum_{j=1}^M g_j(\pmb{\phi}_j) -1\right) & \dot{\lambda} &=& \left(\displaystyle\sum_{i=1}^N f_i(\pmb{\theta}_i) -1\right)  
    \end{Bmatrix} 
\end{equation*}
In the previous theorems we have proved that $(X_{\pmb{\theta}_i(0)},X_{\pmb{\phi}_j(0)})$ are diffeomorphisms. We also know that by definition we have that 
\begin{align*}
    (X_{\pmb{\theta}_i(0)})^{-1}(\pmb{\theta}_i) = f_i(\pmb{\theta}_i) \quad\forall i \in [N]\\
    (X_{\pmb{\phi}_j(0)})^{-1}(\pmb{\phi}_j) = g_j(\pmb{\theta}_i) \quad\forall j \in [M]\\
\end{align*}
We can thus define the following diffeomorphism
\[\nu:
\begin{Bmatrix}
\begin{aligned}
&f_i &=& \quad (X_{\pmb{\theta}_i(0)})^{-1}(\pmb{\theta}_i) \quad\forall i \in [N]\\
&b_j &=& \quad (X_{\pmb{\phi}_j(0)})^{-1}(\pmb{\phi}_j) \quad\forall j \in [M]\\
&\mu &=& \quad \mu\\
&\lambda &=& \quad \lambda\\
\end{aligned}
\end{Bmatrix}
\]
Applying the transform we get a new dynamical system, whose flow we will call $\Phi_{\textrm{distributional}}$:
\begin{equation*}
\Sigma_{\textrm{distributional}}:
 \begin{Bmatrix}
    \dot{f_i} &=& -\norm{\nabla f_i(X_{\pmb{\theta}_i(0)} (f_i))}^2 \left(\sum_{j=1}^M u_{i,j} g_j + \lambda \right)\\
    \dot{g_j} &=&  \norm{\nabla g_j(X_{\pmb{\phi}_j(0)} (g_j))}^2 \left(\sum_{i=1}^N u_{i,j}f_i + \mu \right)\\
    \dot{\mu} &=&  -\left(\sum_{j=1}^M g_j -1\right) \\
    \dot{\lambda} &=& \left(\sum_{i=1}^N f_i -1\right)  
\end{Bmatrix}
\end{equation*}

Although $\Phi_{\textrm{distributional}}$ could be well defined for a wider set of points, we will focus our attention on the following set of points
\begin{equation*}
    \begin{aligned}
        V = & f_{1_{\pmb{\theta}_1(0)}} \times \cdots \times f_{N_{\pmb{\theta}_N(0)}} \\
        &\times g_{1_{\pmb{\phi}_1(0)}} \times \cdots \times g_{M_{\pmb{\phi}_M(0)}} \\
        &\times (-\infty, \infty) \times (-\infty,\infty)
    \end{aligned}
\end{equation*}

Observe that this choice is not problematic since: 
\begin{claim}
$V$ is an invariant set of $\Phi_{\textrm{distributional}}$
\end{claim}
\begin{proof}

Let 
\begin{equation*}
   \pmb{\mathfrak{D}}(t)=(f_1(t),\cdots,f_N(t), g_1(t),\cdots,g_M(t)) 
\end{equation*}
be the profile of all mixed strategies of all agents.
Assume that there is a $t_{\textrm{critical}} \in \mathbb{R}$ such that starting from $\pmb{\mathfrak{D}}_0$, it holds that for some $i\in [N]$, it holds that $f_{i}$ crosses the boundary of $V$ at time $t_{\textrm{critical}}$. Let us call the crossing point $\pmb{\mathfrak{D}}_{\textrm{critical}}$. Since $f_{i}(t_{\textrm{critical}})$ is an end-point of $f_{i_{\pmb{\theta}_{i}(0)}}$ we have that 
\begin{equation*}
    \nabla f_{i} (X_{\pmb{\theta}_{i}(0)}(f_{i}(t_{\textrm{critical}}))) = 0 
\end{equation*}
and thus by the equations of $\dot{f_{i}}$, we have $\dot{f_{i}} = 0$. On the one hand, observe that for $\Phi_{\textrm{distributional}}(\pmb{\mathfrak{D}}_{\textrm{critical}}, \cdot)$ we have that $f_{i}$ should be constant. On the other hand, for $\Phi_{\textrm{distributional}}(\pmb{\mathfrak{D}}_{0}, \cdot)$ it is not the case since $\pmb{\mathfrak{D}}_{0} \in V$ and $\pmb{\mathfrak{D}}_{\textrm{critical}}$ has an $f_{i}$ that is on the edge of $f_{i_{\pmb{\theta}_i(0)}}$. Thus $\Phi_{\textrm{distributional}}(\pmb{\mathfrak{D}}_{0} , \cdot)$ and $\Phi_{\textrm{distributional}}(\pmb{\mathfrak{D}}_{\textrm{critical}}, \cdot)$ are different. This is a contradiction since $\pmb{\mathfrak{D}}_{\textrm{critical}}$ and $\pmb{\mathfrak{D}}_{0}$ belong to the same trajectory of the flow. The same argument applies for $g_j$.

\end{proof}

Clearly $\Phi_{\textrm{original}}( \{\pmb{\theta}_i(0),\pmb{\phi}_j(0), \mu(0), \lambda(0)  \}, \cdot)$ and $\Phi( \{ f_i(\pmb{\theta}_i(0)), g_j(\pmb{\phi}_j(0)), \mu(0), \lambda(0)  \}, \cdot)$ are diffeomorphic. It thus remains to prove that $\Phi$ is Poincar\'{e} recurrent. 

\paragraph{Divergence Free Topological Conjugate Dynamical System}
We will transform the above dynamical system to a divergence free system 
on different space via the following map :
\[\gamma:
\begin{Bmatrix}
\begin{aligned}
a_i&=&\mathcal{A}_i(f_i)&=&  \displaystyle\int_{p_i}^{f_i}\frac{1}{\norm{\nabla f_i(X_{\pmb{\theta}_i(0)} (z))}^2} \mathrm{d} z \quad\forall i \in [N]\\
b_j&=&\mathcal{B}_j(g_j)&=& \displaystyle\int_{q_j}^{g_i}\frac{1}{\norm{\nabla g_j(X_{\pmb{\phi}_j(0)} (z))}^2} \mathrm{d} z \quad\forall j \in [M]\\
\mu &=& \mu \\
\lambda &=& \lambda \\
\end{aligned}
\end{Bmatrix}
\]

\begin{claim}\label{claim:diffeomophism}
$\gamma$ is a diffeomorphism.
\end{claim}
\begin{proof}
Indeed, 
\begin{align*}
 \mathbb{F}_i(f) = \frac{1}{\norm{\nabla f_i(X_{\pmb{\theta}_i(0)} (f_i))}^2}\\
 \mathbb{G}_j(g) = \frac{1}{\norm{\nabla g_j(X_{\pmb{\phi}_i(0)} (g_j))}^2}
\end{align*}
are positive and smooth functions. Thus $\mathcal{A}_i(f_i),\mathcal{B}_j(g_j)$
are monotone functions and consequently bijections and are continuously differentiable. Again because of the monotonicity using Inverse Function Theorem we can show easily that $\mathcal{A}_i(f_i),\mathcal{B}_j(g_j)$ have also continuously differentiable inverse.
\end{proof}

As a first step let us apply $\gamma$ on the equations of our dynamical system:
\begin{equation*}
\begin{aligned}
    \dot{a_i}&=&  \frac{d \mathcal{A}_i(f_i)}{d f_i} \dot{f_i} = \dot{f_i}\frac{1}{\norm{\nabla f_i(X_{\pmb{\theta}_i(0)} (f_i))}^2} =- \left(\sum_{j=1}^M u_{i,j} g_j + \lambda \right)\\
    \dot{b_j}&=&   \frac{d \mathcal{B}_j(g_j)}{d g_j}  \dot{g_j}= \dot{g_j}\frac{1}{\norm{\nabla g_j(X_{\pmb{\phi}_j(0)} (g_j))}^2}  =  \left(\sum_{i=1}^N u_{i,j}f_i + \mu \right) 
\end{aligned}  
\end{equation*}
Observe that on the right hand side of our equations, $f_i$ can be written as $\mathcal{A}_i^{-1} (a_i)$ and $g_j$ can be written as $\mathcal{B}_j^{-1} (g_j)$, so this is an autonomous dynamical system, whoose flow we will call $\Psi$ and whose vector field we will call $\pmb{Y}$:
\begin{equation*}
\Sigma_{\textrm{Preserving}}:
\begin{Bmatrix}
    \dot{a_i}&=& -\left(\sum_{j=1}^M u_{i,j} \mathcal{B}_j^{-1} (g_j) + \lambda \right) &&
    \dot{b_j}&=& \left(\sum_{i=1}^N u_{i,j}\mathcal{A}_i^{-1} (a_i) + \mu \right)\\
    \dot{\mu}&=& - \left(\sum_{j=1}^M \mathcal{B}_j^{-1} (g_j) -1\right) &&
    \dot{\lambda}&=& \left(\sum_{i=1}^N \mathcal{A}_i^{-1} (a_i) -1\right)  
\end{Bmatrix}\Leftrightarrow
\end{equation*}

\begin{equation*}
\Sigma_{\textrm{Preserving}}:\begin{pmatrix}
\dot{a_i}\\ 
\dot{b_j}\\ 
\dot{\mu}\\ 
\dot{\lambda}
\end{pmatrix}=\pmb{Y}\begin{pmatrix}{a_i},{b_j}, {\mu}, {\lambda}\end{pmatrix}
\end{equation*}

Taking the Jacobian of $\pmb{Y}$, all elements across the diagonal are zero : The coordinate of $\dot{a_i}$ does not depend on $a_i$ and the same goes for all state variables. Given that the divergence of the vector field is equal to the trace of the Jacobian, we are certain that this new dynamical system is divergence free:
\begin{equation*}
    \mathrm{div} [\pmb{Y}] = 0
\end{equation*}

Once again we focus our attention on $\gamma(V)$ that is invariant for $\Psi$. To prove this invariant, assume that one trajectory of $\Psi$ starting from inside $\gamma(V)$ escaped it. Then given that $\gamma$ is a diffeomorphism, the corresponding trajectory of $\Phi$ will start from $V$ and also escape it, which is not possible since $V$ is invariant for $\Phi$.

\paragraph{Boundness of Trajectories}
In the next section of the proof, we will show that the trajectories of $\Psi$ are also bounded. Our analysis will be based on the invariant function of Theorem \ref{theorem:invariant_multi}. Note that based on the way we proved Theorem \ref{theorem:invariant_multi}, the invariant supplied there is binding for \textbf{all initializations} in $V$ and not just the trajectory of $\Phi( \{ f_i(\pmb{\theta}_i(0)), g_j(\pmb{\phi}_j(0)), \mu(0), \lambda(0)  \}, \cdot)$.

We will split our proof in two cases.
\begin{claim}
For all initializations in $\gamma(V)$, it holds that $\lambda(t),\mu(t)$ are bounded.
\end{claim}
\begin{proof}
Observe the following fact
\begin{equation*}
    \lambda(t) \to \pm \infty \Rightarrow \int_{\lambda^*}^{\lambda(t)} (z-\lambda^*) \mathrm{d}z \to \infty \Rightarrow H \to \infty
\end{equation*}
The last step of this analysis comes from the fact that $H$ is a sum of non-negative terms so if one of them goes to infinity the whole sum becomes unbounded.
Since initializations in $V$ start with finite values of $H$, it is necessary that $\lambda$ remains bounded. Obviously, the same proof strategy applies to the case of $\mu(t)$. 
\end{proof}

Now let us analyze the rest of the variables
\begin{claim}
For all initializations in $\gamma(V)$, it holds that $a_i(t),b_j(t)$ are bounded.
\end{claim}
\begin{proof}
By definition
\begin{equation*}
    a_i(t) \to \pm \infty \Rightarrow \int_{p_i}^{f_i(t)}\frac{1}{\norm{\nabla f_i(X_{\pmb{\theta}_i(0)} (z))}^2} \to \pm \infty
\end{equation*}
Observe also that
\begin{equation*}
    \int_{p_i}^{f_i(t)}\frac{1}{\norm{\nabla f_i(X_{\pmb{\theta}_i(0)} (z))}^2} \to \pm \infty \Rightarrow \int_{p_i}^{f_i(t)}\frac{z-p_i}{\norm{\nabla f_i(X_{\pmb{\theta}_i(0)} (z))}^2} \to \infty
\end{equation*}
This is true because $z-p_i$ is bounded away from zero when $f_i$ is converging to the edges of $f_{i_{\pmb{\theta}_i(0)}}$ as $p_i$ is in the interior of the set for safe initializations.  Thereofe we can once again conclude that
\begin{equation*}
    a_i(t) \to \pm \infty \rightarrow H \to \infty
\end{equation*}
Once again for initializations in $V$, $H$ remains constant and finite. Therefore $a_i$ should be bounded. The same analysis works for $b_j$.
\end{proof}

\paragraph{Application of Poincar\'{e} Recurrence Theorem}
To summarize the properties that we have established until now , we have shown that system of $\Psi$ is divergence free and has only bounded orbits. Liouville's formula also yields that $\Psi$ is a  volume preserving flow. By applying Poincar\'{e} Recurrence Theorem ( \theoref{theorem:poincar- Recurrence} ) almost all initial conditions in $\gamma(V)$ of $\Psi$ are recurrent. Thus the set $W$ of all non-recurrent points in $\Psi$ has measure zero. 

Using the properties of diffeomorphism, we can to propagate the recurrence behavior of $\Psi$ back to $\Phi_{\textrm{disitributional}}$ using \lemref{lem:congugacy-recurrence}
Thus the set of recurrent points of $\Phi$ is $\gamma^{-1}(W)$. Since diffeomorphisms preserve measure zero sets and $W$ has measure zero, the set of recurrent points of $\Phi$ has measure zero, indicating that $\Phi$ is indeed recurrent.
\end{proof}

\begin{theorem}[Restated \theoref{theorem:sigmoid-poincare-recurrence-behavior}]
Let $f_i$ and $g_j$ be sigmoid functions. Then the flow of  Equation \ref{eq:eq_gda_multi} is Poincar\'{e} recurrent. The same holds for all functions $f_i$ and $g_j$ that are one to one functions and for which all initializations are safe.
\end{theorem}
\begin{proof}
One can notice that since $f_i$ and $g_j$ are invertible functions $X_{\theta_i(0)}(\cdot)$ is totally independent of the choice $\theta_i(0)$. In other words we can substitute 
\begin{align*}
    X_{\theta_i(0)}(\cdot)= {f_i}^{-1}(\cdot)\\
    X_{\phi_j(0)}(\cdot) =  {g_j}^{-1}(\cdot) \\
\end{align*}
Thus, in contrast to the previous theorem (Theorem \ref{theorem:differomorphism-poincare-recurrence}), the construction of $\Phi_{\textrm{distributional}}$ does not depend on the initialization. There is a unique $\Phi_{\textrm{distributional}}$ for all initializations. In fact using the same map $\nu$ as in the previous theorem, we can prove that $\Phi_{\textrm{original}}$ is diffeomorphic to $\Phi_{\textrm{distributional}}$.  However, using the previous theorem  the flow $\Phi_{\textrm{distributional}}$ is Poincar\'{e} recurrent. Repeating the topological conjugacy argument of the previous theorem we can transfer the Poincar\'{e} recurrence property from the dynamical system of $\Phi_{\textrm{distributional}}$ to the dynamical system of $\Phi_{\textrm{original}}$.
\end{proof}

\section[Omitted Proofs of Section \ref{section:spurious}: Spurious equilibria]{Omitted Proofs of Section \ref{section:spurious}\\ Spurious equilibria}

\begin{theorem}[Restated \theoref{theorem:spurious}]\label{restated-theorem:theorem:spurious}
One can construct functions $f$ and $g$ for the system of Equation \ref{eq:eq_gda} so that for a positive measure set of initial conditions the trajectories converge to fixed points that do not correspond to equilibria of the hidden game.
\end{theorem}
\begin{proof}
Our strategy is to analyze the structure of the Jacobian of the vector field of Equation \ref{eq:eq_gda} at stationary points of $f$ and $g$. Let us call $\pmb{Y}(\pmb{\theta}, \pmb{\phi})$ the vector field of Equation \ref{eq:eq_gda}. Now we can write down its Jacobian
\begin{equation*}
    \mathrm{D}\pmb{Y}(\pmb{\theta}, \pmb{\phi}) =
\begin{pmatrix}
    -v \left( g(\pmb{\phi}) -q \right) \nabla^2 f(\pmb{\theta}) &  -v \nabla f(\pmb{\theta}) \otimes \nabla g (\pmb{\phi}) \\
    v \nabla g (\pmb{\phi}) \otimes \nabla f(\pmb{\theta})   & v \left( f(\pmb{\theta}) - p  \right) \nabla^2 g(\pmb{\phi}) 
\end{pmatrix}
\end{equation*}
Let us focus our attention on stationary points of $f$ and $g$. Let us call them $\pmb{\theta}^*$ and $\pmb{\phi}^*$
\begin{equation*}
      \mathrm{D}\pmb{Y}(\pmb{\theta}^*, \pmb{\phi}^*) =
v \begin{pmatrix}
    - \left( g(\pmb{\phi}^*) -q \right) \nabla^2 f(\pmb{\theta}^*) &  \mathbf{0}_{n \times m} \\
    \mathbf{0}_{m \times n}   & \left( f(\pmb{\theta}^*) - p  \right) \nabla^2 g(\pmb{\phi}^*) 
\end{pmatrix}  
\end{equation*}
 We want to study the cases where all eigenvalues of this matrix are negative (i.e. the fixed point is stable). Let $\lambda_i (\nabla^2 f(\pmb{\theta}^*))$ be the eigenvalues of $\nabla^2 f(\pmb{\theta}^*)$ and $\lambda_i(\nabla^2 g(\pmb{\phi}^*))$ the corresponding eigenvalues of $\nabla^2 g(\pmb{\phi}^*)$. Then we know that the eigenvalues of $\mathrm{D}\pmb{Y}(\pmb{\theta}^*, \pmb{\phi}^*)$ are
\begin{align*}
    - v\left( g(\pmb{\phi}^*) -q \right) \lambda_i (\nabla^2 f(\pmb{\theta}^*)) && v\left( f(\pmb{\theta}^*) - p  \right) \lambda_i(\nabla^2 g(\pmb{\phi}^*))
\end{align*}
Here we will analyze the case of $v>0$ (the case of $v<0$ is completely similar). To get that all eigenvalues are negative we can simply require:
\begin{itemize}
    \item $\nabla^2 f(\pmb{\theta}^*)$ and $\nabla^2 g(\pmb{\phi}^*)$ are invertible.
    \item $\pmb{\phi}^*$ is a local minimum with $g(\pmb{\phi}^*)> q$. Combined with the first condition we get that $\nabla^2 g(\pmb{\phi}^*)$ is positive definite.
    \item $\pmb{\theta}^*$ is a local minimum with $f(\pmb{\theta}^*) < p$. Combined with the first condition we get that $\nabla^2 f(\pmb{\theta}^*)$ is positive definite.
\end{itemize}
One can observe that the second condition allows the existence of unsafe initializations if $\pmb{\phi}(0)$ is in the vicinity of $\pmb{\phi}^*$. 

Clearly based on Theorem \ref{theorem:smt-contiuous}, there is a full dimensional manifold of points that eventually converge to this fixed point. Given that the manifold has full dimension, this set of points has positive measure. Additionally, $g(\pmb{\phi}^*)$ and $f(\pmb{\theta}^*)$ do not take the values of the unique equilibrium of the hidden Game.  
\end{proof}

\section[Omitted Proofs of Section \ref{section:discrete}: Discrete Time Gradient-Descent-Ascent]{Omitted Proofs of Section \ref{section:discrete}\\ Discrete Time Gradient-Descent-Ascent}

\shadowbox{
\begin{minipage}[c]{5in}
The outline of this Section is the following:
\begin{enumerate}
    \item We first review an existing result that shows that invariants of continuous time systems that have convex level sets, even though they may not be invariants for the discrete time counterparts, they are at least non-decreasing for the discrete case.
    \item We show that the invariant of Theorem \ref{theorem:invariant_multi} is convex for the case of sigmoid functions. Therefore it has convex level sets.
    \item We extend the construction of Theorem \ref{theorem:spurious} to discrete time systems. 
\end{enumerate}
\end{minipage}
}

\begin{theorem}[Theorem 5.3. of \cite{hamiltonianes}]
\label{restated-theorem:increasingenergy} 
     Suppose a continuous dynamic $y(t)$ has an invariant energy $H(y)$. If $H$ is continuous with convex sublevel sets then the energy in the corresponding discrete-time dynamic obtained via Euler’s method/integration is non-decreasing.
\end{theorem}
\begin{proof}
Let us consider a continuous time dynamical system:
\begin{equation*}
    \derivativeclosed{\pmb{y}(t)} = \pmb{F}(\pmb{y}(t))
\end{equation*}
Let $t$ denote the current time instant of a trajectory with initial conditions $\pmb{y}_0$. Doing discrete time gradient-descent-ascent  with with step-size $\eta$ yields an approximation of $\pmb{y}_{\pmb{y}_0}(t+\eta)$
\begin{align}
	\hat{\pmb{y}_{\pmb{y}_0}}^{t+\eta}= \pmb{y}_{\pmb{y}_0}(t)+\eta \derivativeclosed{\pmb{y}(t)}
\end{align}

To prove our theorem it suffices to show that
\begin{equation*}
    H(\hat{\pmb{y}_{\pmb{y}_0}}^{t+\eta})\geq H(\pmb{y}_{\pmb{y}_0}(t))
\end{equation*}
 
 Suppose $H(\pmb{y}_{\pmb{y}_0}(t))=c$ and without loss of generality, assume $\{\pmb{y}_{\pmb{y}_0}: H(\pmb{y}_{\pmb{y}_0})\leq c\}$ is full-dimensional. 
Since $\{\pmb{y}_{\pmb{y}_0}: H(\pmb{y}_{\pmb{y}_0})\leq c\}$ is convex, there exists a supporting hyperplane $\{\pmb{y}_{\pmb{y}_0}: a^\intercal \pmb{y}_{\pmb{y}_0} = a^\intercal \pmb{y}_{\pmb{y}_0}(t)\}$ such that $a^\intercal \pmb{y}_{\pmb{y}_0} \leq a^\intercal \pmb{y}_{\pmb{y}_0}(t)$ for all $\pmb{y}_{\pmb{y}_0}\in \{\pmb{y}_{\pmb{y}_0}: H(\pmb{y}_{\pmb{y}_0})\leq c\}$.
	
	Because of the invariance property of $H$ over the trajectory with it holds:
	$H(\pmb{y}_{\pmb{y}_0}(t))=c \  \forall t\in \mathbb{R}$
	
	Therefore, 	
	\begin{align*}
		a^\intercal \left(\frac{d}{dt}\pmb{y}_{\pmb{y}_0}(t)\right)&= a^\intercal \left(\lim_{s\to 0^+} \frac{\pmb{y}_{\pmb{y}_0}(t)-\pmb{y}_{\pmb{y}_0}(t-s)}{s}\right)\\
		&= \left(\lim_{s\to 0^+} \frac{a^\intercal \pmb{y}_{\pmb{y}_0}(t)-a^\intercal \pmb{y}_{\pmb{y}_0}(t-s)}{s}\right)\\
		&\geq \left(\lim_{s\to 0^+} \frac{a^\intercal \pmb{y}_{\pmb{y}_0}(t)-a^\intercal \pmb{y}_{\pmb{y}_0}(t)}{s}\right)=0,
	\end{align*}
	implying
	\begin{align*}
		a^\intercal \hat{\pmb{y}_{\pmb{y}_0}}^{t+\eta}&=a^\intercal \pmb{y}_{\pmb{y}_0}(t)+a^\intercal\left(\eta \frac{d}{dt}\pmb{y}_{\pmb{y}_0}(t)\right)\\	
        &\geq a^\intercal \pmb{y}_{\pmb{y}_0}(t).
	\end{align*}
	
	For contradiction, suppose $H(\hat{\pmb{y}_{\pmb{y}_0}}^{t+\eta})< c$. By continuity of $H$, 
for sufficiently small $\epsilon>0$, $\hat{\pmb{y}_{\pmb{y}_0}}^{t+\eta}+\epsilon a \in \{\pmb{y}_{\pmb{y}_0}: H(\pmb{y}_{\pmb{y}_0})\leq c\}$.  
	However, 
	\begin{align}
		a^\intercal (\hat{\pmb{y}_{\pmb{y}_0}}^{t+\eta}+\epsilon a) \geq a^\intercal \pmb{y}_{\pmb{y}_0}(t) +\epsilon||a||_2^2 > a^\intercal \pmb{y}_{\pmb{y}_0}(t)
	\end{align}
	contradicting that $\{\pmb{y}_{\pmb{y}_0}: a^\intercal \pmb{y}_{\pmb{y}_0} = a^\intercal \pmb{y}_{\pmb{y}_0}(t)\}$ is a supporting hyperplane. 
	Thus, the statement of the theorem holds.
\end{proof}

\begin{lemma} \label{lemma:h-convexity}
The invariant of Theorem \ref{theorem:invariant_multi} is jointly convex in $\pmb{\theta}$, $\pmb{\phi}$, $\lambda$ and $\mu$ when $f_i$ and $g_j$ are sigmoid functions of one variable. 
\end{lemma}
\begin{proof}
Since $H$ is a sum of terms each involving disjoint variables, it suffices to prove that each term is convex with respect to its own variables. This follows immediately for $\lambda$ and $\mu$. Let us take one term involving $f_i$ (the same analysis works for $g_j$ terms as well). In fact we want to prove that the following function is convex
\begin{equation*}
    \int_{p_i}^{f (\theta_i)} \frac{z- p_i}{\norm{ \nabla f (X_{\theta_i (0) } (z) ) }^2}\mathrm{d}z
\end{equation*}
where $f$ is the sigmoid function. Taking the first derivative, knowing that $f' = (1-f) f$ for sigmoid we have
\begin{equation*}
    \frac{\left( f (\theta_i)- p_i \right) \left( 1- f (\theta_i) \right) f(\theta_i)}{\norm{ \nabla f (X_{\theta_i (0) } (f(\theta_i)) ) }^2}
\end{equation*}
$X_{\theta_i (0) } (f(\theta_i))$ is equal to $\theta_i$ since $f$ is one-to-one. Thus we can simplify
\begin{equation*}
    \frac{\left( f (\theta_i)- p_i \right) \left( 1- f (\theta_i) \right) f(\theta_i)}{\norm{ \nabla f (\theta_i) }^2}
\end{equation*}
Once again we can use the formula for the derivative of $f$
\begin{equation*}
    \frac{\left( f (\theta_i)- p_i \right) \left( 1- f (\theta_i) \right) f(\theta_i)}{\left(\left( 1- f (\theta_i) \right) f(\theta_i)\right)^2} =
    \frac{\left( f (\theta_i)- p_i \right)}{\left( 1- f (\theta_i) \right) f(\theta_i)} 
\end{equation*}
In order to complete the convexity analysis we must take the second derivative test. 
\begin{equation*}
    \frac{d}{d \theta_i } \frac{\left( f (\theta_i)- p_i \right)}{\left( 1- f (\theta_i) \right) f(\theta_i)}  =
    \frac{f(\theta_i)^2 -2p_i f(\theta_i) + p_i}{\left( 1- f(\theta_i) \right)^2 f(\theta_i)^2}   \left( 1- f(\theta_i) \right) f(\theta_i) = 
    \frac{f(\theta_i)^2 -2p_i f(\theta_i) + p_i}{\left( 1- f(\theta_i) \right) f(\theta_i)}
\end{equation*}
The only roots of the numerator are
\begin{equation*}
    f(\theta_i) = p_i \pm \sqrt{p_i^2 - p_i}
\end{equation*}
Of course for $p_i \in (0,1)$ these roots are not real. So for all $\theta_i$, $f(\theta_i) \in (0,1)$ and the second derivative is positive. This concludes our convexity proof. 
\end{proof}

\begin{figure}[ht!]
        \centering
        \includegraphics[width=\textwidth]{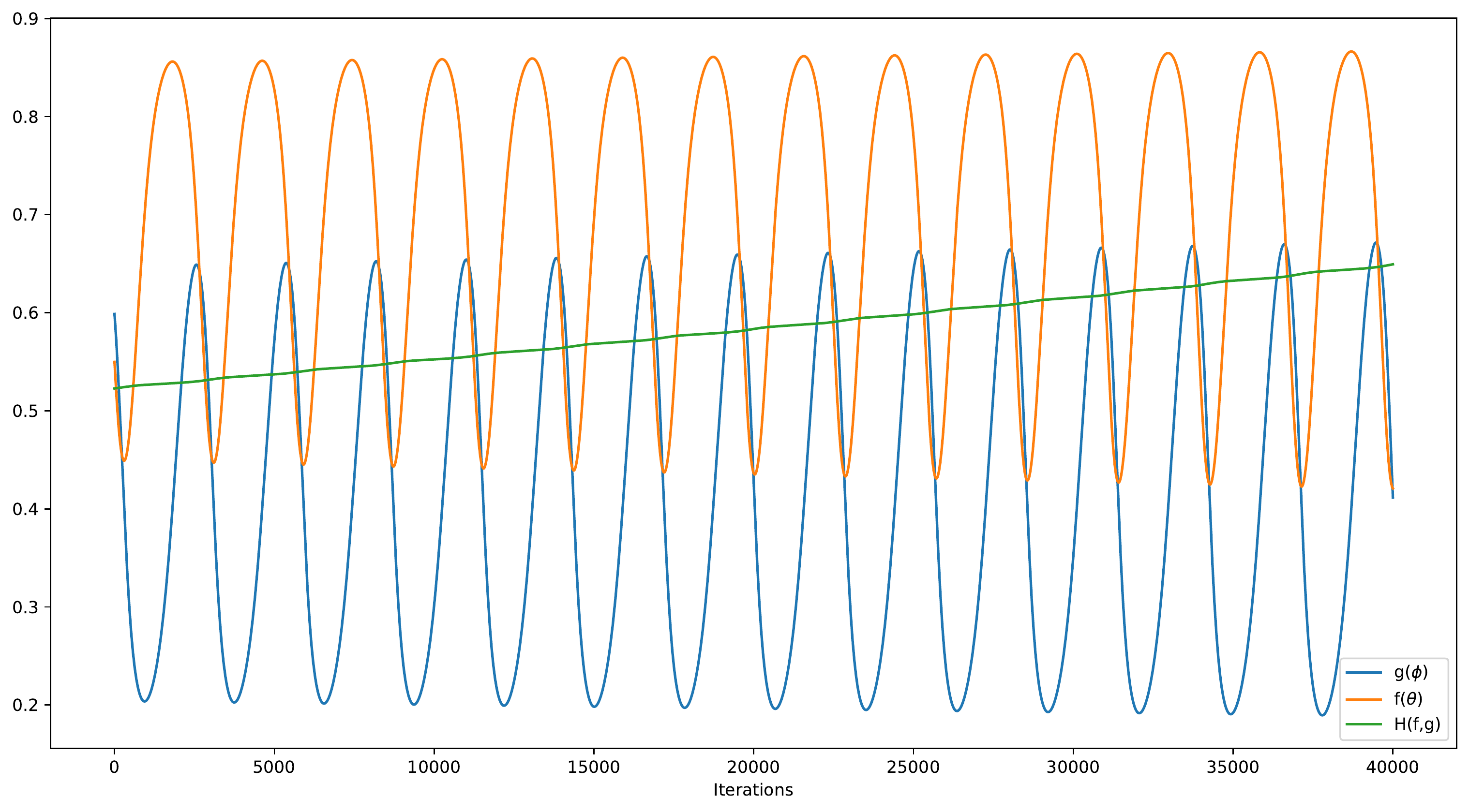}
        \caption{Hidden bilinear game with two strategies having $p=0.7$ and $q=0.4$. The functions $f$ and $g$ are sigmoids for each player. We observe the evolution of $f$ and $g$ as well as the invariant of Theorem \ref{theorem:invariant}. The trajectories are close to being periodic but $H$ has began to increase even with relatively few iterations, confirming the findings of Theorem \ref{theorem:discrete-hamiltonian}.}
        \label{fig:increaseing-hamiltonian}
\end{figure}

\begin{theorem} [Restated \theoref{theorem:discrete-hamiltonian}]
\label{restated-theorem:discrete-hamiltonian}
Let $f_i$ and $g_j$ be sigmoid functions. Then for the discretized version of the system of Equation \ref{eq:eq_gda_multi} and for safe intializations, function $H$ of Theorem \ref{theorem:invariant_multi} is non-decreasing.
\end{theorem}
\begin{proof}
First observe that given that sigmoids are invertible functions so $X_{\theta_i(0)}(f_i)$ and $X_{\phi_j(0)}(g_j)$ are independent of the initial conditions similar to the proof of Theorem \ref{theorem:sigmoid-poincare-recurrence-behavior}. Thus invariant of Theorem \ref{theorem:invariant_multi} $H$ preserved by all the trajectories of the continuous time dynamical system is common across all initializations. Using Lemma \ref{lemma:h-convexity}, $H$ is convex and therefore has convex level sets. Of course it is also continuous. Using Theorem \ref{restated-theorem:increasingenergy} we get the requested result.
\end{proof}

\begin{theorem}[Restated \theoref{theorem:spurious-discrete}]
\label{restated-theorem:spurious-discrete}
One can choose a learning rate $\alpha$ and functions $f$ and $g$ for the discretized version of the system of Equation \ref{eq:eq_gda} so that for a positive measure set of initial conditions the trajectories converge to fixed points that do not correspond to equilibria of the hidden game.
\end{theorem}
\begin{proof}
The proof follows the same construction as in the continuous case of Theorem \ref{theorem:spurious}. In fact, the Jacobian of the discrete time map is
\begin{equation*}
    I_{(N+M)\times (N+M)} + \alpha \mathrm{D}\pmb{Y}(\pmb{\theta}, \pmb{\phi})
\end{equation*}
where $\pmb{Y}$ is the vector field of the continuous time system. We can do the same construction as in Theorem \ref{theorem:spurious}, to get a fixed point $(\pmb{\theta}^*, \pmb{\phi}^*)$ such that $\mathrm{D}\pmb{Y}(\pmb{\theta}, \pmb{\phi})$ has only negative eigenvalues and $(f(\pmb{\theta}^*), g(\pmb{\phi}^*)) \neq (p,q)$. Let $\lambda_{min}$ be the smallest eigenvalue of this matrix. Choose
\begin{equation*}
    \alpha < - \frac{1}{\lambda_{min}}
\end{equation*}
Then the Jacobian of the discrete time map has positive eigenvalues that are less than one. Therefore the discrete time map is locally a diffeomorphism and by the Stable Manifold Theorem for discrete time maps (Theorem \ref{theorem:smt-discrete}), the stable manifold is again full dimensional and therefore has positive measure.  
\end{proof}

\begin{figure}[!ht]
        \centering
        \includegraphics[width=\textwidth]{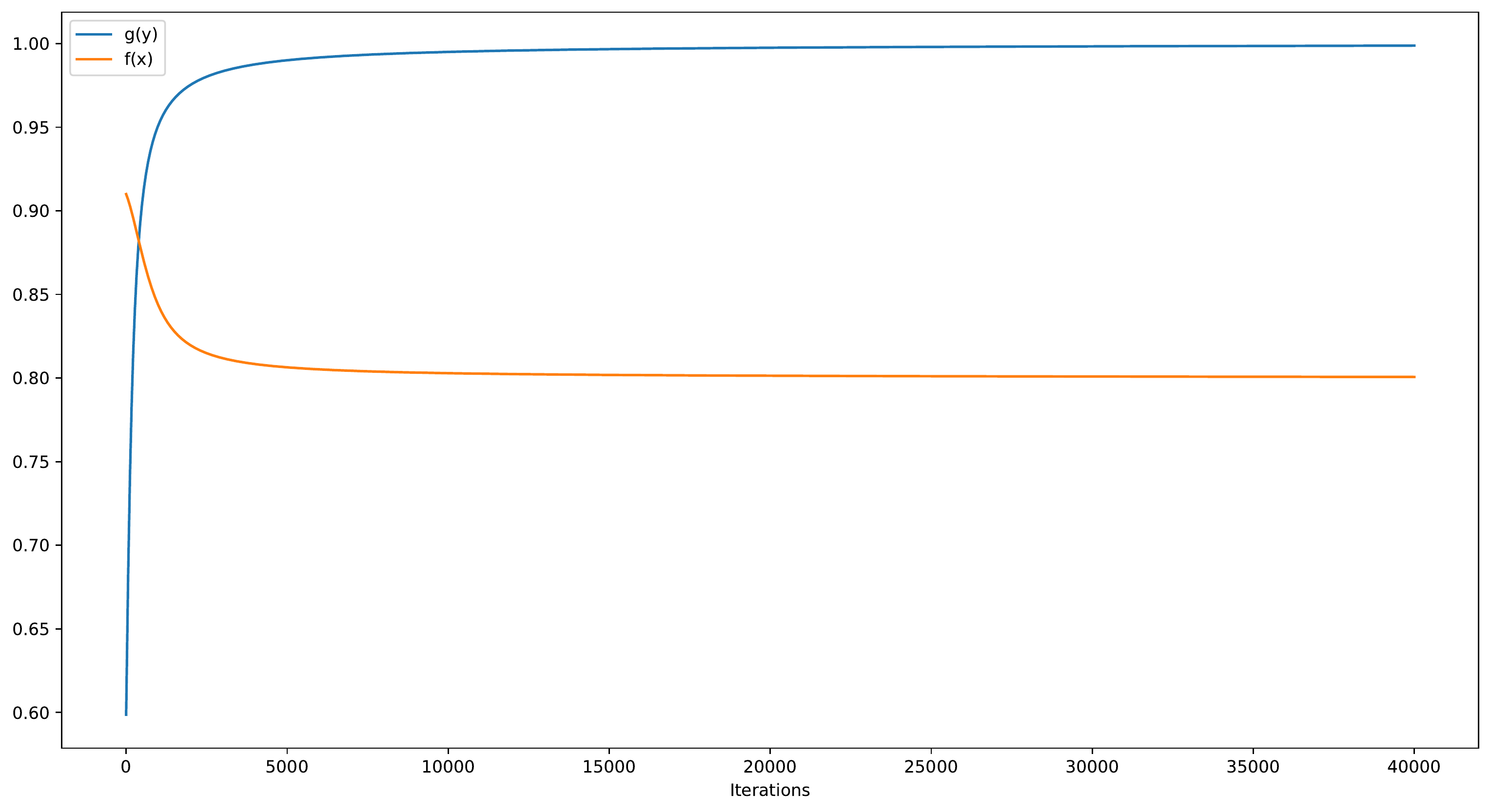}
        \caption{Hidden bilinear game with two strategies having $p=0.4$ and $q=0.2$. The functions are $f(x) = 0.8 + 0.2 \cdot \sigma(x)$ and $g(y) = \sigma(y)$. There is no solution of $f(x)=p$ and therefore no initialization is safe. The dynamical system converges to an equilibrium that is not game theoretically meaningful, verifying the findings of Theorem \ref{theorem:spurious-discrete}.}
        \label{fig:spurious}
\end{figure} \end{document}